\documentclass[10pt,a4paper,final,oneside]{article} 
\usepackage{latexsym, amsfonts, amsthm, amsmath, amssymb, tikz, enumitem, booktabs, float,color,multicol,epsf, url,epsfig,epstopdf, array, setspace, graphicx, xcolor,tikz-cd}
\usepackage[margin=1in,footskip=0.4in]{geometry}
\usepackage{pinlabel}
\usepackage{soul,xcolor}
\usepackage{amsfonts}
\usepackage{amsmath}
\usepackage{amssymb}
\usepackage[applemac]{inputenc}
\usepackage{url}
\usepackage{hyperref}
\usepackage{csquotes}
\usepackage{lipsum}
\usepackage{accents}
\usepackage{mathtools}
\usepackage[T1]{fontenc}
\usepackage{tikz-cd}
\usepackage{url} 
\usepackage{hyperref} 
\usepackage{authblk}

\usepackage{pdfpages}
\usepackage{graphicx}

\definecolor{greenbf}{rgb}{0.2, 0.8 ,0.4}

\def\R{\mathbb{R}}

\def\H{\mathbb{H}}
\def\Z{\mathbb{Z}}
\def\mS{\mathbb{S}}

\def\T{\mathbb{T}}
\def\M{\mathcal{M}}
\def\F{\mathcal{F}}

\def\Isom{{\operatorname{Isom}}}

\def\Aff{\operatorname{Aff}}

\def\Aut{\sf{Aut}}

\def\heis{{\mathfrak{heis}} }

\def\Aut{{\sf{Aut}}}

\def\Heis{{\sf{Heis}}}

\def\O{{\operatorname{O}}}
\def\SO{{\operatorname{SO}}}

\def\Stab{{\sf{Stab_1}} }
\def\Stab{{\sf{Stab}} }

\def\Span{{\sf{span}} }

\def\Mink{{\operatorname{Mink}} }

\def\tr{{\sf{tr}}}

\def\Ker{{\sf{Ker}}}

\def\p{{\mathfrak{p}}}

\def\a{{\mathfrak{a}}}
\def\g{{\mathfrak{g}}}

\def\so{{\mathfrak{so}}}

\def\s{{\mathfrak{s}}}
\def\h{{\mathfrak{h}}}

\def\z{{\mathfrak{z}}}
\def\z{{\mathfrak{z}}}

\def\FF{{\mathfrak{F}}}
\def\H{{\mathfrak{H}}}
\def\E{{\mathfrak{E}}}

\def\la{{\langle}}
\def\ra{{\rangle}}
\def\la{{\langle}}
\def\ra{{\rangle}}
\def\X{{\mathbf{X}}}

\newcommand{\ad}{\mathrm{ad}}

\def\Heis{{\operatorname{Heis}}}

\def\Sol{{\operatorname{Sol}}}
\def\Osc{{\operatorname{Osc}}}
\def\L{{\operatorname{L}}}

\def\Span{{\sf{Span}}}
\def\G{{\widehat{G}}}
\def\M{{\widetilde{M}}}

\def\V{{\mathcal{V}}}

\def\cd{{\mathsf{cd}}}

\def\bX{\mathbf{X}}

\newtheorem{theorem}{{Theorem}}[section]
\newtheorem{pr}[theorem]{{Proposition}}
\newtheorem{isom.ext}[theorem]{{Trivial isometric extension}}
\newtheorem{defi}[theorem]{{Definition}}
\newtheorem{lemma}[theorem]{{Lemma}}
\newtheorem{sublem}[theorem]{{Sub-Lemma}}
\newtheorem{cor}[theorem]{{Corollary}}
\newtheorem{fact}[theorem]{{\sc Fact}}
\newtheorem{remark}[theorem]{{Remark}}

\newtheorem{question}[theorem]{{Question}}
\newtheorem{example}[theorem]{{Example}}
\newtheorem{conv}[theorem]{{Convention}}
\newtheorem{examples}[theorem]{{Examples}}

\newtheorem{notation}[theorem]{Notation}

\definecolor{purple}{rgb}{0.65,0.12,0.94}
\definecolor{forestgreen}{rgb}{0.4,0.64,0.13}
\newcommand{\interior}[1]{%
  {\kern0pt#1}^{\mathrm{o}}%
}

\hypersetup{
	colorlinks=true,
	linkcolor=black,
	filecolor=magenta,      
	urlcolor=violet,
    citecolor=violet,}

\setstcolor{red}

\title{\Large{\textbf{\textsc{On completeness of certain locally symmetric pseudo-Riemannian manifolds of signature $(2,2)$}}}}

\author{Malek Hanounah}

\date{}

\begin{document}

\maketitle

\begin{abstract}
We show geodesic completeness of certain compact locally symmetric pseudo-Riemannian manifolds of signature $(2,n)$. Our model  space $\bX$ is a $1$-connected, indecomposable symmetric space of signature $(2,n)$, that admits a unique (up to scale) parallel lightlike vector field. This class of spaces is the natural generalization of the class of Cahen--Wallach spaces to signature $(2,n)$. In dimension $4$ we show that $\bX$ has no proper domain $\Omega$  which is divisible by the action of a discrete group $\Gamma$ of $\Isom(\bX)$, i.e. $\Gamma$ acts properly and cocompactly on $\Omega$. Therefore, we deduce geodesic completeness in the aforementioned situation.  In arbitrary dimension we show geodesic completeness of compact locally symmetric space modeled on $\X$ under the assumption that the transition maps of $M$ are restrictions of transvections of $\bX$. Along the way, we establish a new case in the Kleinian $3$-dimensional Markus's conjecture for flat affine manifolds. Moreover, we classify geometrically Kleinian compact manifolds that are modeled  on the hyperbolic oscillator group endowed with its bi-invariant metric. Finally we discuss a natural dynamical problem motivated by the Lorentz setting (Brinkmann spacetimes).  Specifically, we show that the parallel flow on $M$ is equicontinuous in dimension $4$, even in our non-Lorentz setting.

\end{abstract}
\subsubsection*{Statements and Declarations}
\paragraph{\textbf{Competing interests:}} The authors have no competing interests to declare that are relevant to the content of this article. \\\textbf{Data availibility} Data sharing is not applicable to this article as no new data were created or analyzed in this study. 

\tableofcontents
\section{Introduction}\label{Section: Intro}
Let $M$ be a compact locally symmetric pseudo-Riemannian manifold. The question whether $M$ is geodesically complete has been studied across the literature by various authors focusing mainly on the Lorentz signature. 
The interest in  the completeness question can be justified in two big points: 
\begin{itemize}
    \item[1)] The question relates the local property on $M$ of being locally symmetric, i.e. parallel curvature tensor and  the global geometry that is the geodesic flow is complete.
    \item[2)] Completeness reduces the difficult classification problem of compact manifolds modeled on a certain symmetric space $X$ to the ``simpler'' algebraic problem of understanding discrete groups $\Gamma\subset \Isom(X)$ that act properly cocompactly and freely on $X$.
\end{itemize}

In particular the second point answers the question: which topological manifold $M$ can be endowed with a certain $(G,X)$-geometry.

\subsection{Brief history} Every compact Riemannian manifold (symmetric or not) is geodesically complete by Hopf-Rinow, therefore the completeness problem in the Riemannian situation is completely understood. However going up in signature just by one, increases the difficulty drastically. In the constant curvature case a deep theorem by Carri\`ere and Klingler \cite{carriere1989autour,klingler1996completude} shows that a compact Lorentz manifold of constant sectional curvature is geodesically complete. On the other hand, it is well known that indecomposable Lorentz symmetric spaces of dimension bigger than one are either of non-zero constant curvature or isometric to a \textit{Cahen--Wallach space} \cite{CahenWallach}. Recently, Leistner and Schliebner \cite{leistner2016completeness} showed geodesic completeness of compact manifolds locally isometric to a Cahen--Wallach spaces, establishing in particular that indecomposable compact locally symmetric Lorentz manifolds are geodesically complete. Very recently the geodesic completeness also of all  decomposable compact locally symmetric Lorentz manifolds has been confirmed \cite{allout2025completeness}.

The geodesic completeness of compact locally symmetric spaces of higher signature, namely, of signature $(2,n)$, is much less studied and much more complicated. Essentially nothing is known! The only existing result is due to  Tholozan (to the best of our knowledge) \cite{tholozan2014completeness}, who proves the completeness of compact flat Hermite-Lorentz manifolds (it is a reduced flat geometry), under the assumption that the structure is \textit{Kleinian}, i.e.  the manifold is a quotient of an open domain of the flat model by the action of a discrete subgroup of isometries that acts properly. The open domain of a Kleinian structure is also called a \textit{divisible} domain  \cite{benoist2008survey}.

In this paper the aim is to motivate the completeness problem in the non-Lorentz signature. This is done via natural first examples, which turned out to be quite subtle even in low dimensions due to lack of general techniques. More precisely, we consider compact (indecomposable) locally symmetric spaces of signature $(2,n)$ and focusing on $n=2$.

\subsection{Symmetric pp-waves of signature $(2,n)$}
Lorentz pp-waves are Lorentz manifolds that are defined by a dynamical-geometrical property, namely, they admit a global parallel lightlike vector field $V$ such that the leaves of the totally geodesic codimension one foliation tangent to $V^\perp$ are flat with respect to the induced connection. The class of Cahen--Wallach spaces is properly contained in the class of pp-waves, namely, Cahen--Wallach spaces are indecomposable symmetric pp-waves. One can extend the definition of pp-waves to higher signature in the following way.

\begin{defi}
 Let $\bX$ be a simply connected pseudo-Riemannian space of signature $(p,q)$ and $V$ be a parallel lightlike vector field on $\bX$. We say that $V$ defines a pp-wave structure on $\bX$ if the leaves of the parallel orthogonal distribution $V^\perp$ are flat with respect to the induced connection.
\end{defi}

\subsubsection{Rank $1$ pp-waves of signature $(2,n)$}
A fundamental difference between Lorentz (symmetric) pp-waves and higher signature variants lies in the fact that there can be more than one parallel lightlike vector field!  Despite the latter remark, there are symmetric pp-waves in signature $(2,n)$ which naturally admit only one (up to scale) parallel field  $V$. We want to distinguish these spaces by the following definition.  

\begin{defi}(Rank $1$ pp-waves)
    Let $\bX$ be a simply connected pp-wave of signature $(p,q)$. We say that $\bX$ is a rank $1$ pp-wave if $\bX$ admits a unique (up to scale) pp-wave structure.
\end{defi}
\begin{remark}
    Non-flat Lorentz pp-waves are rank $1$ pp-waves of signature $(1,n)$ according to the above definition. 
\end{remark}

 Symmetric pp-waves of signature $(2,2)$ are described explicitly in the recent work of Kath and Lyko \cite[Subsect. 3.3 and 3.4]{kath2024pseudo}. Their description is based on the classification result by Kath and Olbrich \cite{kath2009structure}. 
 
 It is worth noting that even in dimension $4$, there is a large class of infinitely many non-isometric symmetric spaces of signature $(2,2)$ (see Subsec. \ref{Subsec: application 4d}). In contrast to the Riemannian situation where there are only finitely many in each dimension.
\begin{remark}[Symmetric ``rank $2$'' pp-waves of signature $(2,n)$]    In signature $(2,n)$ there are concrete examples of $4$-dimensional of symmetric pp-waves that are not of rank $1$. They are indecomposable symmetric spaces that admit two linearly independent pp-waves structures, moreover, the isometry group does not preserve none of these structures, rather it preserves only the parallel lightlike plane field they span! The plane field gives rise to a totally geodesic, flat totally isotropic foliation by surfaces (see \cite[Subsect. 3.4]{kath2024pseudo}).
\end{remark}
When a manifold $M$ is locally isometric to a symmetric rank $1$ pp-wave of signature $(2,n)$, we say simply that \textit{$M$ is a locally symmetric rank $1$ pp-wave of signature $(2,n)$}. We show

\begin{theorem}\label{Intro: Theorem 4d}
    Let $M$ be a Kleinian compact locally symmetric rank $1$ pp-wave of signature $(2,2)$. Then $M$ is geodesically complete.
\end{theorem}

In higher dimensions we show that compact models of the ``transvection geometry'' are complete, 

\begin{theorem}\label{Intro: Theorem transvections}
    Let $M$ be a compact locally symmetric rank $1$ pp-wave of signature $(2,n)$. Assume that $M$ has an atlas such that the transition maps of the charts are restrictions of transvections. Then $M$ is geodesically complete.
\end{theorem}
As an immediate corollary we get,
\begin{cor}
     Let $M$ be a compact locally symmetric rank $1$ pp-wave of signature $(2,n)$ modeled on $\X$. Assume that the isometry group of the $\X$ is virtually equal to the transvection group. Then $M$ is geodesically complete.
\end{cor}
\subsection{Applications} We list some important results and consequences that arise from the main theorem.
\subsubsection{On the $3$-dimensional Markus's conjecture}
Let $M$ be a compact locally symmetric rank $1$ pp-wave $\bX$ of signature $(2,n)$.  Since $V$ is unique (up to scale), the manifold $M$ admits a well defined codimension one lightlike flat foliation, so, the leaves of the foliation are natural flat affine manifolds, i.e. endowed with an $(\Aff(\R^k),\R^k)$-structure. Restricting to when $M$ is $4$-dimensional and assuming that the flat lightlike leaves are closed, we get that the leaves are, in fact, endowed with a (sub)-affine geometry with a non-unipotent structural group \begin{equation}\label{eq: unimodular group}
    \L_u(1,1):=\left\{ \begin{pmatrix}
    1 & \star & \star \\
    0 & e^t &0 \\
    0 & 0& e^{-t}
\end{pmatrix} | \ t\in \R \right\}\ltimes \R^3.
\end{equation}
We show, 
\begin{theorem}\label{Markus3d}
    Let $N$ be a  compact Kleinian flat affine $3$-manifold. Assume that the linear holonomy of $N$ is (up to conjugacy) a subgroup of   \begin{equation*}
        \left\{ \begin{pmatrix}
    1 & \star & \star \\
    0 & e^t &0 \\
    0 & 0& e^{-t}
\end{pmatrix} | \ t\in \R \right\}.
    \end{equation*} Then $N$ is complete.
\end{theorem}
The above statement cannot be directly deduced from known literature, since the discompacity of the linear holonomy is \(2\), Carri\`ere's theorem \cite{carriere1989autour} does not apply. Also, neither \cite{fried1981affine} nor \cite{fried1986distality} can be applied, since the holonomy is neither unipotent nor \textit{distal} (for definition \cite[Section 1]{fried1986distality}). It is also worth noting that closed affine manifolds with such holonomy exist (with non-trivial hyperbolic part) \cite[Theorem 4.1]{fried1981affine}.

\subsubsection{Hyperbolic oscillator geometry}
In dimension $4$, there are exactly two Lie groups that admit a non-flat bi-invariant pseudo-Riemannian metric. They are the oscillator group and the hyperbolic oscillator group. The oscillator group has a bi-invariant Lorentz metric (which is indecomposable), therefore it is a Cahen--Wallach space. Compact manifolds locally modeled on the oscillator geometry are well understood. On the other hand, the hyperbolic oscillator group $\Osc_s$ has a bi-invariant metric of signature $(2,2)$ and it is globally isometric to a symmetric rank $1$ pp-wave of signature $(2,2)$. As a corollary of Theorem \ref{Intro: Theorem 4d} we classify compact Kleinian manifolds modeled on the hyperbolic oscillator geometry, i.e. compact Kleinian $(\Isom(\Osc_s), \Osc_s)$-manifolds, as follows:

\begin{theorem}\label{Intro: cor hyperbolic oscillator}
    Let $M$ be a compact Kleinian $(\Isom(\Osc_s), \Osc_s)$-manifold.  Then $M$ is isometric to a quotient of the hyperbolic oscillator group (endowed with its bi-invariant metric) by a cocompact lattice or $M$ is isometric to a quotient $\Gamma \backslash (\R\times\Heis_3)$, where $\Gamma$ is a cocompact lattice of $\R\times\Heis_3$ (with a suitable left-invariant metric). 
\end{theorem}
The above theorem shows, in particular, that all compact Kleinian manifolds locally modeled on the hyperbolic oscillator are \textit{standard}, i.e. the group $\Gamma$ is a cocompact lattice in a connected Lie subgroup $S$ of the isometry group that acts properly in the model space.
\begin{remark}(Lattices of the hyperbolic oscillator)
    As it appears in the above theorem, one part of the dichotomy is that the manifold $M$ is isometric to a quotient of the hyperbolic oscillator group by a (cocompact) lattice. A complete classification of lattices of the hyperbolic oscillator group is given in \cite{kathsplitosci}, so a complete isometric classification of compact Kleinian manifolds modeled on the hyperbolic oscillator geometry seems ``achievable''.
\end{remark}

The last part discusses some dynamics related to the parallel flow on the compact locally symmetric rank $1$ pp-waves of signature $(2,2)$.
\subsubsection{Dynamics of the parallel flow}
Let $(M,g)$ be a compact pseudo-Riemannian manifold, and let $W$ be a Killing field of $M$, i.e. the flow of $W$ is isometric. One question we could ask is whether the action $W$ is equicontinuous. Equicontinuity is equivalent to the flow of $W$ preserving a Riemannian metric, so in a sense it is dynamically trivial. For Lorentz locally symmetric pp-waves equicontinuity of the parallel lightlike field is conjectured to be true. In the indecomposable case, i.e. $M$ is a Cahen--Wallach space, the lightlike vector field $W$ is shown  \cite[Proposition 8.2]{kath-CW} to be periodic. In the higher signature case, equicontinuity of the parallel flow is not expected to hold in every dimension. However, surprisingly, we show that it does hold in dimension 4, exhibiting similar behavior to that in the Lorentz setting.

\begin{theorem}\label{Intro: Theorem Equicontnuity}
    Let $M$ be a $4$-dimensional compact Kleinian locally symmetric rank $1$ pp-wave. Then the parallel flow of $M$ is equicontinuous.
\end{theorem}
\subsection{Open questions}
In the light of Theorem \ref{Intro: Theorem 4d} one can expect that a similar theorem also holds in higher dimensions. For that it seems one has to a show a general version of Theorem \ref{Markus3d}, namely, compact Kleinian $(\L_u(1,n-1),\R^{n+1})$-manifolds are complete, for the definition of the group $\L_u(1,n-1)$ see (\ref{Eq: L_u(1,n-1)}).

Another direction could be to consider compact complete $(\L_u(1,n-1),\R^{n+1})$-manifolds and classify their fundamental group. More precisely,

\begin{question}
    Is the fundamental group of a compact complete $(\L_u(1,n-1),\R^{n+1})$-manifold virtually solvable?
\end{question}

The answer of the above question implies that compact (Kleinian) locally symmetric rank $1$ pp-waves $M$ of signature $(2,n)$, with closed lightlike leaves have virtually solvable fundamental group. 
\subsection{Organization of the article} In Section \ref{Section: Prelim} we present the basic material that we use throughout the paper. Section \ref{Section: geo of models} is devoted to exploiting the isometric invariant objects of  symmetric rank $1$ pp-waves of signature $(2,n)$. We use the latter geometries in Section \ref{Section: completeness transvections} to show Theorem \ref{Intro: Theorem transvections}.  Then in Section \ref{Section: completeness 4d} we peruse the completeness question of compact manifolds modeled on $4$-dimensional symmetric rank $1$ pp-waves with essentially larger isometry group than the transvections, namely the models are the hyperbolic oscillator group with it bi-invariant metric and the ``elliptic'' geometry, however, under the assumption that closed model is a Kleinian, we show Theorem \ref{Intro: Theorem 4d} and Theorem \ref{Markus3d}. Section \ref{Section: topology} is devoted to the geometric classification namely we show Theorem \ref{Intro: cor hyperbolic oscillator}. The last Section is a comment on the dynamics of the parallel vector field for $4$-dimensional compact locally symmetric rank $1$ pp-waves. It contains the proof of Theorem \ref{Intro: Theorem Equicontnuity}.
\subsection*{Acknowledgment}
This work would not be possible without the help of my advisor Ines Kath. 
I am also thankful for her continuous support and encouragement.

\section{Geometric structures and symmetric spaces}\label{Section: Prelim}

We briefly recall the notion of geometric structures on manifolds.

\subsection{Geometric structures}\label{subsection: GX}

Let $G$ be a Lie group and $X$ a space on which $G$ acts smoothly transitively, then one has $X=G/H$, where $H$ is a closed subgroup of $G$. A $(G,X)$-structure on a manifold $M$ is an $X$-valued atlas on $M$ where the chart transitions are restrictions of elements of $G$. A manifold with a $(G, X)$-structure is called a $(G, X)$-manifold. Having a $(G,X)$-manifold $M$ gives rise to a developing pair $(D, \rho)$. Namely, $D$ is a local diffeomorphism from the universal cover $\M$ into \textit{the model space} $X$, while $\rho$ is a morphism from the $\pi_1(M)$ into the \textit{structural group} $G$. These maps satisfy the following equivariance formula 
 \begin{equation}
     D(\gamma x) =\rho(\gamma)D(x),
 \end{equation} for any $x\in \widetilde{M}$ and $\gamma \in \pi_1(M)$.
The map $D$ is called a \textit{developing map} and $\rho$ is a \textit{holonomy representation}. 

More on the theory of $(G,X)$-structures can be found in  \cite{goldman2022geometric} and  \cite{thurston2022geometry}.

\subsubsection{G-invariant objects}\label{G invariant} We recall a simple yet important principle for $(G,X)$-structures. Having a $G$-invariant object, e.g. a tensor, connection, etc., on the the model $X$ gives rise to a well defined corresponding object on any $(G,X)$-manifold $M$ ``\textit{a pullbacked object}''. The process is done via pullbacking the object locally using local charts, this gives rise to a well defined object due to $G$-invariance.
 
\begin{defi}[Completeness in the sense of $(G, X)$-structures]
    Let $M$ be a $(G, X)$-manifold, we say that $M$ is $(G,X)$-complete if the developing map $D: \widetilde{M} \to X$ is a covering map.
\end{defi}
The next fact justifies the above definition, when the model space has a $G$-invariant connection.
\begin{fact}[Proposition 3.7 \cite{thurston2022geometry}]\label{Fact: complete=complete}
    Let $M$ be $(G,X)$-manifold. Assume that $X$ is complete with respect to a $G$-invariant connection. Then $M$ is geodesically complete (with respect to the ``pullback'' connection) if and only if $M$ is $(G,X)$-complete.
\end{fact}
\begin{example}[The hyperbolic oscillator geometry]\label{Example: SplitOsc}
    The hyperbolic oscillator (or sometimes called the split oscillator) group $\Osc_s$ is  isomorphic to $\R\ltimes \Heis_3$ where $\R$ acts on Heisenberg fixing its center and by hyperbolic transformations on the Heisenberg group $\Heis_3$ modulo its center.
    This group has a (non-flat) bi-invariant metric of signature $(2,2)$. This is analogous to its Lorentz counterpart the oscillator group $\Osc$, which is isomorphic to $\R\ltimes \Heis_3$, where $\R$ acts by fixing the center and by elliptic transformations the quotient modulo center. Explicitly, the metric of the hyperbolic oscillator is given in the basis $\la T,X,Y,Z\ra$ where $\heis_3$ is spanned by $X,Y,Z$ with $[X,Y]=Z$ by the following \begin{equation}
        \la T,Z\ra=1, \ \la X,Y\ra=1.
    \end{equation} Since the metric is bi-invariant, the center action defines a parallel lightlike vector field $V$ and the leaves of the orthogonal distribution $V^\perp$ are flat. Thus, the hyperbolic oscillator is, in fact, an indecomposable symmetric rank $1$ pp-wave of signature $(2,2)$.  In other words, the hyperbolic oscillator geometry is the $(G,X)$-geometry where $X=\Osc_s$ and $G=\Isom(\Osc_s)$.\end{example}

    \begin{remark}
        The hyperbolic oscillator geometry is a particular bi-invariant geometry. A bi-invariant geometry is defined via a Lie group $L$ that admits a bi-invariant pseudo-Riemannian metric, e.g. when $L$ is semi-simple. Then the bi-invariant geometry is the $(L\times L/\Delta(Z),L)$-geometry, where $\Delta(Z)$ is the diagonal embedding of the center of $L$. For $L=\operatorname{PSL}_2(\R)$ the corresponding bi-invariant geometry is the $3$-dimensional anti de-Sitter geometry.
    \end{remark}
\subsection{Locally symmetric spaces as $(G,X)$-manifolds}\label{Subsection: locally symmetric GX} Let $M$ be a locally symmetric pseudo-Riemannian  manifold. Then $M$ is locally isometric to a simply connected symmetric pseudo-Riemannian space $X$. This is due to the fact, that any local isometry on a simply connected symmetric pseudo-Riemannian space is the restriction of a global isometry (\cite{o1983semi}), it follows that $M$ is naturally a $(\Isom(X),X)$-manifold, in the sense of $(G,X)$-structures (see Subsec. \ref{Subsection: locally symmetric GX}, for more details see \cite[chapter 3]{thurston2022geometry} or \cite{goldman2022geometric}).

On the other hand, an important subgroup of the isometry group is the transvection group, denoted by $\G_X$.  The transvection group is in some special cases ``essentially'' smaller than the full isometry group, that is, they have different dimensions. As in the example of the flat space $\R^{2,n}$, the full isometry group is $\O(2,n)\ltimes \R^{n+2}$, the transvection group is the group $\R^{n+2}$ of pure translations. Considering locally symmetric spaces that are modeled on the transvection geometry, i.e. $(\G_X,X)$-manifolds is a decent first step to understand the geometry of compact manifolds locally modeled on $X$.

\subsection{Foliated geometric structures}

 A smooth manifold $M$ with a foliation $\mathfrak{F}$ is called a foliated manifold.
Here we present the notion of\textit{ tangential $(G,X)$-foliations}. Basically it is a family of $(G,X)$-structures on the leaves that vary continuously from one leaf to another. Later we will show that our symmetric spaces are naturally $(G,X)$-foliated manifolds for a specific $G$ and $X$.
\begin{defi}[\textbf{Tangential} \textbf{$(G,X)$-foliations}]\label{Defi: GX foliation}
We say that a manifold $M$ has a tangential $(G,X)$-foliation if $M$ is a foliated manifold such that the leaves of $\mathfrak{F}$ are modeled on $(G,X)$. More precisely, for every $p\in M$ there exist an open neighborhood  $U$ and  a diffeomorphism (into the image) $\varphi: U\to  \R^k\times X$ such that every leaf of the foliation of $U$ (induced by $\mathfrak{F}$) is mapped to a level $(t,V_t)\subset \R^k\times X$.  Moreover, if $(\varphi_i, U_i)$ and $(\varphi_j, U_j)$ are two such maps, the transition map $\varphi_{ij}:=\varphi_j\varphi_i^{-1}$ has the form $(t,x)\mapsto(\psi(t), g_t(x))$, where $g_t\in G$ and $x \in V_t$ (compare \cite{inaba1993tangentially, hanounah2025completeness}).
\end{defi}

\begin{conv}
    Throughout the paper by a $(G,X)$-foliated manifold we mean a tangentially $(G,X)$-foliated manifold unless otherwise stated.
\end{conv}
Now we give a concrete example of a tangentially $(G,X)$-foliated manifold, via our favorite example, namely the hyperbolic oscillator geometry.

\begin{example}[The hyperbolic oscillator as a $(G,X)$-foliated manifold]\label{Example: SplitOsc foliated}
    The full isometry group of the hyperbolic oscillator group with its bi-invariant metric is virtually equal to the group of left and right multiplication (both are isometric since the metric is bi-invariant), namely, up to finite index \begin{equation*}
        \Isom(\Osc_s)\cong (\Osc_s\times \Osc_s)/\Delta(Z),
    \end{equation*}
    where $\Delta(Z)$ is the diagonal embedding of the center of $\Osc_s$. The center of $\Isom(\Osc_s)$ gives an $\Isom(\Osc_s)$-invariant parallel lightlike vector field $V$, we denote by $\FF$ the flat foliation tangent to $V^\perp$. The stabilizer (in $\Isom(\Osc_s)$) of one  $\FF$-leaf is virtually isomorphic to $\SO^{\circ}(1,1)\ltimes \Heis_5$. Since the leaves are flat with respect to a torsion free connection, they are natural affine manifolds, i.e. $(\Aff(\R^3), \R^3)$-manifolds. However, the structure group for the leaves is reduced to a smaller group. It is given by the stabilizer of one leaf. Indeed, the geometry of an $\FF$-leaf can be described as the maximal affine geometry with two invariant objects: a parallel vector field $W$ and a parallel degenerate Lorentz metric with radical equal to $\R W$. The action of  $\SO^{\circ}(1,1)\ltimes \Heis_5$ on any leaf (affinely equivalent to $\R^3$) is given by the affine group $\L_u(1,1)$. Hence, the oscillator group with its bi-invariant metric is a natural $(\L_u(1,1), \R^3)$-foliated manifold.

\end{example}
\section{Geometry of symmetric rank \texorpdfstring{$1$}{1} pp-waves of signature \texorpdfstring{$(2,n)$}{(2,n)}} \label{Section: geo of models}
Let $\bX$ be a rank $1$ (indecomposable) symmetric pp-wave of signature $(2,n)$ with $\dim(\bX)=n+2$. The goal of this section is devoted to investigate in depth the geometric invariants of $\bX$.
\subsection{Transvection group}
The first step in understanding these spaces is via their transvection group, i.e. the group generated by transvection elements, which is denoted by $\G_\bX$.
\begin{pr}\label{Prop: structure of the transvections}
    Let $\bX$ be a symmetric rank $1$ pp-wave of signature $(2,n)$ with $\dim(\bX)=n+2$.  Then the Lie algebra of the transvection group of $\bX$ is given by $\widehat{\g}_\bX\cong \R\ltimes_L\heis_{2n+1}$, where $\R$ acts trivially on the center of $\heis_{2n+1}$ via a non-inner derivation $L$. Moreover, the symmetric space  $\bX$ is isometric to the homogeneous quotient $\G_\bX/I$ where $I$ is a rank $n$ abelian subgroup of $\Heis_{2n+1}$ transversal to the center. The center of Heisenberg generates the lightlike vector field $V$ and the lightlike leaves are the $\Heis_{2n+1}$-orbits.
\end{pr}
A proof of this fact can be found in the classification by Kath and Olbrich \cite[Theorem 7.10 item (1),(3)-(4)]{kath2009structure}. For more details on the transvection algebra and for a self-contained  proof see Appendix \ref{Appendix}.
\subsection{Blend of geometries}
 We denote the identity component of the isometry group $G_\bX:=\Isom^{\mathsf{o}}(\bX)$, since $\bX$ is simply connected and symmetric, it follows by Gromov's theory of rigid transformation groups that $[\Isom(\bX): G_\bX]$ is finite \cite[Theorem 3.5.A]{gromovrigid}. 
\subsubsection{Parallel lightlike vector field}
The first invariant is the unique (up to scale) parallel lightlike vector field $V$. Due to Proposition \ref{Prop: structure of the transvections} the vector field $V$ is $\G_\bX$-invariant, however, we show it is even $G_\bX$-invariant. Let $\varphi\in G_\bX$ be any element, up to multiplying $\varphi$ by an appropriate $\psi\in \G_\bX$ we can assume that $\varphi$ fixes a point $o$. Thus it is enough to show that $V$ is $I(o)$-invariant where $I(x)$ is the isotropy group of $x$ in $G_\bX$. Now let $\varphi$ and consider its differential $\varphi^*$. Then, the map $\varphi^*$ is an isometric automorphism of the transvection algebra $\widehat{\g}_\bX$ endowed with its bi-invariant metric that comes from the metric on $\bX$ (see \cite[Lemma 2.3]{kath2024pseudo}). Since $V$ is unique up to scale we know that $\varphi^*V=\lambda V$ for some non-zero $\lambda$. However, $\varphi^*$ preserves the metric $\la\cdot, \cdot\ra$ on $\widehat{\g}_\bX\cong \R L\ltimes \heis_{2n+1}$. Since $\la L,V\ra=1$, we deduce that $\varphi^*$ induces an automorphism of $\heis_{2n+1}$ and \begin{equation*}
    \varphi^*L=\frac{1}{\lambda}L \ \mathrm{mod} \ \heis_{2n+1}.
\end{equation*} Moreover, $\varphi^*[L,A]=[\varphi^*L,\varphi^*A]$ for all $A\in \heis_{2n+1}$. As, $[\heis_{2n+1},\heis_{2n+1}]\subset V$ we get \begin{equation*}
    \varphi^*L(A)=\frac{1}{\lambda}L\varphi^*(A) \ \mathrm{mod} \ V,
\end{equation*} for all $A\in \heis_{2n+1}$. Hence, $\lambda\varphi^*L=L\varphi^*$.  On the other hand, by Proposition \ref{Prop: structure of the transvections} we know that $L$ is non-inner, in particular, $L$ is non-zero. Thus, $\lambda=\pm 1$. 
Since, $I(o)$ is connected, we can assume that $\lambda=1$.

Since $V$ is $G_\bX$-invariant, we get a parallel $G_\bX$-invariant orthogonal distribution $V^\perp$, which gives rise to a $G_\bX$-invariant, totally geodesic lightlike foliation $\FF$.  Each leaf $\F$ of the foliation is foliated by the orbits of the $V$-flow, the foliation given by the $V$-orbits is denoted by $\mathcal{V}$.
\subsubsection{Codimension one totally geodesic flat foliation $\FF$}
The condition that the $\FF$-leaves are flat endow each leaf with a natural $(\Aff(\R^{n+1}), \R^{n+1})$-structure. Moreover, since $\bX$ is geodesically complete, each totally geodesic leaf is also complete, hence affinely equivalent to $\R^{n+1}$. In this sense, the space $\bX$ is an $(\Aff(\R^{n+1}), \R^{n+1})$-foliated manifold. However, the existence of the parallel degenerate Lorentz metric $h:=-dx_1^2+\sum_{i=2}^ndx_i^2$ written in $(v,x_1,\cdots,x_n)$, where $v$ is the direction of the radical $h$ which is $\R V$, reduces the structural group of this tangential geometry. Namely, in the above  coordinates the structural group is a subgroup of   \begin{equation}\label{Eq: L_u(1,n-1)}
    \L_u(1,n-1):=\left\{\begin{pmatrix}
    1 & \star \\
    0 & A
\end{pmatrix} | \ A \in \O(1,n-1)  \right\}\ltimes \R^{n+1}.
\end{equation} 

 Up to taking the quotient of a leaf $\F/\V$ (i.e. the quotient of $\F$ by the $V$-action) we get the flat Minkowski space $\R^{1,n-1}$. Let us remark here (not very crucial for later), that the $\V$-foliation on each $\FF$-leaf is a transversal Lorentz foliation and $V$-flow is a transversal Lorentz flow (for more details on transversal $(G,X)$-foliations see \cite{blumenthal,carriere1984flots}, for Lorentz foliations see \cite{boubel2006lorentzian}).
\begin{remark}
    The full isometry group of $\bX$ is up to finite index a subgroup of $(\R\times \O(1,n-1))\ltimes \Heis_{2n+1}$.
\end{remark}
\subsubsection{Isometry invariant translation structure on the space of $\FF$-leaves}\label{para: space of leaves str}
Let $M$ be manifold with a codimension one foliation $\mathfrak{G}$. We say that $\mathfrak{G}$ has a transversal affine structure if there is a submersion $h:\M\to \R$ where $\mathfrak{G}$ is given by the levels of $h$, and there is a group morphism $\varphi:\pi_1(M)\to \Aff(\R)$ and $h$ is $\varphi$-equivariant (see \cite{blumenthal}). Returning to our situation, we consider the one form $\eta:=g(V,\cdot)$, where $g$ is the metric of the symmetric space. The form $\eta$ is clearly $G_\bX$-invariant, since both the metric and the vector field are $G_\bX$-invariant, i.e. for all $k\in G_\bX$ we have $k^*\eta=\eta$ which is equivalent to 
\begin{equation}\label{eq: translations}
    f\circ k=f+c_k \ \ \text{where} \ \ c_k\in \R.
\end{equation} Moreover, since both the metric $g$ and the vector field $V$ are parallel $\eta$ is even closed. So, $\eta=df$ where $f:\bX\to \R$ is a submersion, this gives the foliation $\FF$ a transversal translation structure. The space of leaves $\bX/\FF$, is diffeomorphic to $\R$. Indeed, by Proposition \ref{Prop: structure of the transvections} we know that $\bX\cong \R\ltimes \Heis_{2n+1}/I$, where $I\subset \Heis_{2n+1}$. Moreover, the leaves are given by $\Heis_{2n+1}$-orbits. Finally observe that this translation structure on $\bX/\FF$ is by construction $G_\bX$-invariant.

\subsection{Transvection group invariants}\label{Subsection: transvection invariants}
Considering the geometric invariants for the transvection group reduces the complexity of the situation. We claim that the stabilizer of an $\FF$-leaf in the transvection group is a unipotent subgroup of $\L_u(1,n-1)$, which is the Heisenberg group in ``standard'' form,
\begin{equation}
    \left\{\begin{pmatrix}
    1 & \star \\
    0 & I_n
\end{pmatrix} \right\}\ltimes \R^{n+1}\subset \L_u(n):=\left\{\begin{pmatrix}
    1 & \star \\
    0 & A
\end{pmatrix} | \ A \in \O(n)  \right\}\ltimes \R^{n+1}.
\end{equation}

Indeed, by Proposition \ref{Prop: structure of the transvections} the transvection group is $\R\ltimes \Heis_{2n+1}$. The stabilizer of one leaf in $\G_\bX$ is equal to $\Heis_{2n+1}$. On the other hand, $\Heis_{2n+1}$ acts isometrically on $\bX$, so it must act affinely on the $\FF$-leaves. On the Lie algebra level we can decompose $\heis_{2n+1}=\a^+\oplus\a^-\oplus \R V$, where $\a^+$ is the Lie algebra of the isotropy at the point $o$. The equation $[\a^+,[\a^+,\cdot ]]=0$ yields that the isotropy action is unipotent. Now the elements in $\exp(\a^-\oplus V)$ are in fact transvection elements of the symmetric space $\bX$. However, since these transvection elements preserve the flat leaves, they induce transvection elements of the flat affine space, which are just pure translations. The latter are clearly unipotent. To summarize the $\Heis_{2n+1}$-action is unipotent and has the claimed form.

The situation is drastically more difficult, with respect to the completeness problem, if we consider the full isometry group instead of only the transvection group, since the $\V$-foliation is a transversal Lorentz foliation and no longer a Riemannian foliation (as seen in Example \ref{Example: SplitOsc foliated}).

\section{Completeness result for the transvection geometry}\label{Section: completeness transvections}

In this section we show that any compact manifold $M$ locally isometric to a symmetric rank $1$ pp-wave $\bX$ of signature $(2,n)$, such that transitions (in the sense of the geometric structure) are in the transvection group $\G_\bX$ is geodesically complete. Equivalently,
\begin{theorem}\label{Theorem: complete transvection}
    Let $M$ be a compact $(\G_\bX,\bX)$-manifold. Then $M$ is $(\G_\bX,\bX)$-complete. 
\end{theorem}
Recall the principle of $G$-invariant objects in paragraph \ref{G invariant}.
We will introduce the following notation. 
\begin{notation}
    We denote the  objects ``pullbacked'' from the model $\bX$ to $M$ by a subscript $M$, e.g. $V_M$, and the lifted object on the universal cover by a subscript $\widetilde{M}$, e.g. $V_{\M}$.
\end{notation}
\subsection{Strategy of the proof} First, we show that the $\mathfrak{F}_M$-leaves are geodesically complete, i.e. they are developed bijectively onto the affine space. Since $M$ is a compact $(\L_u(n), \R^{n+1})$-foliated manifold, we use the Following Fact \ref{Fact: Lu structure}.
Finally, we complete the proof using an equivariant (with respect to the developing map of $\widetilde{M}$) transverse flow to the $\mathfrak{F}_\M$-leaves. 
\begin{fact}[Theorem 1.9 in \cite{hanounah2025completeness}]\label{Fact: Lu structure}
   Let $(M,\mathfrak{F})$ be a compact $(\L_u(n), \R^{n+1})$-foliated manifold.  The leaves of $\mathfrak{F}$ are $(\L_u(n), \R^{n+1})$-complete.
\end{fact}

 As an immediate consequence we have
\begin{cor}\label{Cor: complete leaves}
     The $\mathfrak{F}_M$-leaves are geodesically complete with respect to the   flat affine connection on the leaves induced by the Levi-Civita connection on $M$.
\end{cor}
\begin{proof}
    The foliated manifold $(M, \mathfrak{F}_M)$ is a natural $(\L_u(n), \R^{n+1})$-foliated manifold (see Subsection \ref{Subsection: transvection invariants}). So, the leaves are $(\L_u(n), \R^{n+1})$-complete by the previous fact (Fact \ref{Fact: Lu structure}). This is equivalent to the geodesic completeness by Fact \ref{Fact: complete=complete}. 
\end{proof}
The final ingredient of our proof is the $\G_\bX$-invariant  one form $\eta=g(V,\cdot)$ (see Paragraph \ref{para: space of leaves str}).
 \begin{lemma}
     Let $\eta$ be a non-singular one-form  on $M$. There exists a vector field $W$ such that $\eta(W)=1$. 
 \end{lemma}
 \begin{proof}
     The locus where the one-form $\eta$ has value one is a subbundle of the affine tangent bundle. Each fiber is given by a translation of the kernel of $\eta$. So, this is a fiber bundle over $M$ with contractible fiber. Hence, there is a section, which is our vector field $W$. \end{proof}
 \paragraph{\textbf{Structure of the universal cover.}} We use the next fact to write the universal cover as a product of one $\mathfrak{F}_\M$-leaf   and the real line.
\begin{fact}\label{Fact: structure of the universal cover}
Let $M$ be a compact manifold endowed with a  nowhere vanishing closed one-form $\eta$. Then the universal cover of $M$ is diffeomorphic to $\R\times \F_\M$ where $\F_\M$ the foliation tangent to the $\Ker(\eta)$. In particular, the $\mathfrak{F}_\M$-leaves are all diffeomorphic and simply connected.
\end{fact}
A proof of the previous fact can be found in \cite[Proposition 8]{leistner2016completeness}.

\begin{proof}[Proof of Theorem $\ref{Theorem: complete transvection}$]
     Fix $(D, \rho)$ a developing pair for the $(\G_\bX,\bX)$-structure on $M$. By Fact \ref{Fact: complete=complete} and because  $\bX$ is simply connected, the   proof is reduced to show that $D$ is a diffeomorphism onto $\bX$.
     Since the foliation $\mathfrak{F}$ is $\G_\bX$-invariant, the restriction of $D$ to any (simply connected) $\mathfrak{F}_\M$-leaf is a developing map for the $(\L_u(n), \R^{n+1})$-geometry. Thus, the restriction of $D$ is bijective (by Corollary $\ref{Cor: complete leaves}$), namely $D_{|\F_\M}: \F_\M\to \R^{n+1}$ is a diffeomorphism. 
     
     Now, let $\eta_M$ (resp. $\eta_\M$) denote the pullbacked one form on $M$ (resp. on $\M$). Let $W$ be a smooth vector field on $M$ such that $\eta_M(W)=1$. The flow of $\widetilde{W}$ clearly preserves the $\mathfrak{F}_\M$-leaves and its flow is complete (i.e. defined for all $t\in \R$) because $M$ is compact. Moreover, due to the fact that $\eta$ is $\G_\bX$-invariant we have the following equivariance formula \begin{equation}\label{Eq: equivariance}
         D(\phi^t_{\widetilde{W}}(\F_\M))=(\phi_{W_\bX}^t\circ D)(\F_\M)=\phi_{W_\bX}^t(\F),
     \end{equation}
     where $\phi^t$ denotes the flow, and $W_\bX$ is any vector field such that $\eta(W_\bX)=1$. We observe that the flow of $\widetilde{W}$ acts freely and transitively on the space of $\mathfrak{F}_\M$-leaves. Now we are ready to show that $D$ is injective on $\M$. For the sake of contradiction assume that there exist two different points $\Tilde{p}_1, \Tilde{p}_2$ such that $D(\Tilde{p}_1)=D(\Tilde{p}_2)$.  Then $\Tilde{p}_1, \Tilde{p}_2$ necessarily belong to different leaves since we already know that the restriction of $D$ to each leaf is injective. However, this means that the two leaves $\Tilde{p}_i\in\F_\M^{t_i}$ are mapped to the same leaf  $\F$ by the developing map. On the other hand we know that $\phi^{t_2-t_1}(\F_\M^{t_1})=\F_\M^{t_2}$. Applying $D$ to both sides and using the equivariance we get \begin{equation}
         \phi_{W_\bX}^{t_2-t_1}(D(\F_\M^{t_1}))=D(\F_\M^{t_2}),
     \end{equation}
     which is equivalent to $\phi_{W_\bX}^{t_2-t_1}(\F)=\F$, this is absurd, since we assumed that $\Tilde{p}_1, \Tilde{p}_2$ belong to different leaves, meaning that  $t_2\neq t_1$, and  the flow of $W_\bX$ acts freely on the space of $\mathfrak{F}$-leaves. 
     
    Having the bijectivity of the developments of the lightlike leaves, it is clear that $D$ is surjective if and only if it induces a bijection from the space of $\FF_\M$-leaves to the space of $\FF$-leaves. However, since $M$ is compact the flow of $W$ is complete, thus, so is $\widetilde{W}$. The equivariance formula (\ref{Eq: equivariance}) then implies the surjectivity of $D$. \end{proof}
  
\subsection{Application in dimension 4: Transvection geometry}\label{Subsec: application 4d}
Recently in \cite{kath2024pseudo} the authors gave an explicit description of $1$-connected indecomposable non-semisimple symmetric spaces of of signature $(2,2)$. In particular, they describe symmetric rank $1$ pp-waves of signature $(2,2)$. In their notation these spaces are given by three families. Each family depends on several continuous and discrete parameters. Different parameters give rise to spaces in different isometry classes, hence, one gets a continuous family of non-isometric symmetric spaces.  We list them below.  Recall the notation $G_\bX=\Isom^{\mathsf{o}}(\bX)$ from Section \ref{Section: geo of models}.
\subsubsection{$\bX_1$-spaces} This family is denoted by $\bX_1(\varepsilon_1, \varepsilon_2, \lambda)$ where $\varepsilon_1, \varepsilon_2 \in \{+1, -1\}$ and $\lambda\in \R_{>0}$. The identity component of the full isometry group is determined by the above parameters in the following way \cite[Proposition 3.11 (i)]{kath2024pseudo}:
\begin{itemize}
    \item \textbf{Essential-type.} This occurs exactly when $\varepsilon_1\neq \varepsilon_2$ and $\lambda= 1$. Then we have an essentially bigger isometry group, namely, $\G_{\bX_1}\subsetneq G_{\bX_1}\cong(\R\times \SO^{\circ}(1,1))\ltimes \Heis_5$.
    \item \textbf{Transvection-type.} In all other cases, we have that $G_{\bX_1}=\G_{\bX_1}\cong\R\ltimes \Heis_5$.
\end{itemize}

The two interesting situations are $ \bX_\H:=\bX_1(1,-1,1)$ and $\bX_\E:=\bX_1(-1,1,1)$, where the isometry group is ``essentially'' larger than the transvection group. The space  $\bX_\H$  is particularly interesting because it is isometric to the hyperbolic oscillator with its bi-invariant metric of signature $(2,2)$ (see example \ref{Example: SplitOsc}). 

\subsubsection{The remaining families of symmetric rank $1$ pp-waves of signature $(2,2)$}
All the other $4$-dimensional symmetric rank $1$ pp-waves are of transvection type,  namely, for any parameter defining the remaining model spaces we have $G_{\bX}=\G_{\bX}\cong \R\ltimes \Heis_5$ (for the proof see \cite[Proposition 3.11 (ii)-(iii)]{kath2024pseudo}).

\subsubsection{Clifford--Klein forms} In addition to \cite{kath2024pseudo} and very recently, Maeta in \cite{Matea} studied Clifford--Klein forms of the latter models. Namely, he characterizes for which symmetric rank $1$ pp-wave $\bX$ of signature $(2,2)$ there exists a discrete subgroup $\Gamma\subset \Isom(\bX)$ which acts properly cocompactly and freely on $\bX$. The quotient $\Gamma\backslash X$ is called a \textit{Clifford--Klein form} of $\bX$. In the language of $(G,X)$-structures, $M$ is a compact complete $(\Isom(\bX),\bX)$-manifold. 

More interestingly, in \cite{kath2024pseudo} the authors show in dimension $4$ that symmetric rank $1$ pp-waves are the only $1$-connected indecomposable non-semisimple symmetric spaces of signature $(2,2)$ that admit Clifford--Klein forms. More precisely, by \cite{Matea,kath2024pseudo} the two spaces $\X_\H$ and $\X_\E$  are the only ones that have Clifford--Klein forms.

Before the next result we recall the following,
\begin{remark}\label{Obser: finite index completenesss}
    Let $G_0\subset G$ be a finite index normal subgroup. Let $M$ be a compact $(G,X)$-manifold. If $(G_0,X)$-manifolds are $(G_0,X)$-complete, then $M$ is $(G,X)$-complete.
\end{remark}
\begin{proof}
Let $(D, \rho)$ be a developing pair.
    Define $\Gamma:=\rho(\pi_1(M))$ and $\Gamma^0:=\Gamma\cap G_0$. We see that $\rho^{-1}(\Gamma^0)$ is a normal subgroup of finite index of $\pi_1(M)$. Let $M_0$ be the corresponding finite cover $M$. Then $M_0$ is naturally a $(G_0, X)$-manifold. Since, $M_0$ is $(G_0,X)$-complete and it covers $M$, it follows that $M$ is $(G,X)$-complete.
 \end{proof}
In connection with the previous results we establish the following. 

\begin{cor}Let $\bX$ be a $4$-dimensional symmetric rank $1$ pp-wave of transvection type, i.e. up to finite index $\G_\bX=\Isom(\bX)$. Then there are no compact manifolds modeled on $\bX$ i.e.  there are no compact $(\Isom(\bX),\bX)$-manifolds.
\end{cor}
\begin{proof}
 Let $\bX$ be one of the spaces mentioned in the statement.  We have that $[\Isom(\bX):G_\bX]<\infty$. Now, let $M$ be a compact $(\Isom(\bX), \bX)$-manifold, by Remark \ref{Obser: finite index completenesss}, $M$ has a finite cover $M_0$ which is $(\G_\bX,\bX)$-manifold. Applying Theorem \ref{Theorem: complete transvection} we deduce that $M_0$ is geodesically complete. Hence $M$ is geodesically complete. It follows that $\Gamma$ acts properly on $\bX$ with the quotient identified to $M$, in other words $M$ is a Clifford-Klein form of $\X$. However, by \cite[Theorem 1.7]{Matea} the above spaces do not admit Clifford--Klein forms.
\end{proof}
\begin{remark}
    The above corollary is a stronger statement than \cite[Theorem 1.7]{Matea}. For, a priori, a compact model is not necessarily a Clifford--Klein form.
\end{remark}

\section{The \texorpdfstring{$4$}{4}-dimensional Kleinian case}\label{Section: completeness 4d}
In this section we pursue the completeness question in dimension $4$, for the two remaining ``essential-type'' spaces, without assuming anything on the transition maps. However, under the assumption that $M$ is Kleinian, i.e. $D$ is a covering map onto a subset $\Omega\subset \bX$ such that the holonomy group $\Gamma$ acts on it properly, cocompactly and freely (see \cite{kulkarni2006uniformization}). 

The remaining spaces are the hyperbolic oscillator (space/group) $\mathbf{X}_\H$ with its bi-invariant metric together with the ``elliptic'' space $\mathbf{X}_\E$. Thanks to Remark \ref{Obser: finite index completenesss} we can reduce the study of completeness of $\mathbf{X}_\H$ and $\mathbf{X}_\E$ to the identity components of their isometry groups $G_\H$ and $G_\E$ respectively.

\subsubsection{The leaf structure}\label{Subsection: leaf structure}
We recall that $\bX \in \{\mathbf{X}_\E, \mathbf{X}_\H\}$ has a codimension one flat foliation $\FF$ tangent to $V^\perp$. The group $G_{\bX}$ acts on the space of $\FF$-leaves by translations.  The stabilizer of an $\FF$-leaf is isomorphic to $\L_u(1,1)\cong\SO^{\circ}(1,1)\ltimes \Heis_5$, see (\ref{eq: unimodular group}), which acts by affine transformations on each leaf.

We recall from Section \ref{Section: completeness transvections}, that we have 
\begin{pr}\label{Prop: space of F-leaves}
    The developing map $D:\M\to\bX$ induces a diffeomorphism from the space of $\mathfrak{F}_\M$-leaves to the space of $\mathfrak{F}$-leaves. 
\end{pr}
\begin{proof}
    The proof is identical to the proof of Theorem \ref{Theorem: complete transvection} (second paragraph).
\end{proof}
So, we get the following
\begin{cor}\label{Cor: leaf complete iff M complete}
     The developing map $D$ is bijective on $\M$ if and only if its restriction to each $\FF_\M$-leaf is bijective.
 \end{cor}
 \begin{proof}
     Observe that the image of each $\mathfrak{F}_\M$-leaf is sent by the developing map to an $\mathfrak{F}$-leaf. If $D$ is bijective on each leaf, then by Proposition \ref{Prop: space of F-leaves} we conclude that $D$ is bijective on the whole universal cover. The other direction is clear.
 \end{proof}

Using Fact \ref{Fact: structure of the universal cover} we see that each $\FF_\M$-leaf is the universal cover of an $\FF_M$-leaf. In particular, the  restriction of $D$ to any leaf $\F_\M$ is a developing map for $\F_M$ seen as $(\L_u(1,1),\R^3)$-manifold.

The goal of the next subsection is to show that the developing of the flat lightlike leaves is bijective onto the affine $3$-space. Then by Corollary \ref{Cor: leaf complete iff M complete} it follows that $\Omega=\bX$.
\subsection{Completeness of  the lightlike leaves} 
 \begin{notation}
     We denote by $\Gamma_1:=\Gamma\cap \L_u(1,1)$ and $\Gamma_0:=\Gamma\cap \Heis_5$. Moreover, we define the natural projection $q: \L_u(1,1)\to \SO^{\circ}(1,1)$, and $\widetilde{\psi}:G_\bX\to G_\bX/Z$.
 \end{notation} We now state the following important lemma.

\begin{lemma}\label{Lemma: center transversal}
Let $Z\subset \L_u(1,1)$ be the center of Heisenberg. Assume that $q(\Gamma_1)$ is non-trivial, then we have $\Gamma_0\nsubseteq  Z$.
\end{lemma}
 \begin{proof}
    The model space $\bX$ has global coordinates of the form $(v,x,y,u)$, for which the $\SO^{\circ}(1,1)$-action is by linear Lorentz transformation on the $(x,y)$-component (see \cite[Proposition 3.9]{kath2024pseudo}). Assume that $\Gamma_0\subset Z$, then up to conjugacy by an element of $\L_u(1,1)$ we can assume that $\Gamma \subset (\R\times \SO^{\circ}(1,1)) \times Z$. Indeed, we denote elements of $G_\bX$ by $(s,t,a,z)$ where $(s,t)\in \R\times\SO^{\circ}(1,1)$ and $(a,z)\in \Heis_5$. Since $q(\Gamma_1)$ is non-trivial, we can find an element of the form $\gamma_0:=(0,t_0,a_0,z_0)$ where $t_0\neq 0$. For $\alpha=(0,0,a_1,0) \in \Heis_5$, we get the following equation \begin{equation}\label{eq: conjuagcy}
        (0,t_0,a_1+a_0-t_0\cdot a_1,*).
    \end{equation} Since the induced action of $\SO^{\circ}(1,1)$ on $\Heis_5/Z$ has eigenvalues $\exp(\pm \lambda)$ for $\lambda>0$, we can solve Equation (\ref{eq: conjuagcy}) by choosing $a_1$ so that  $a_1-t_0\cdot a_1+a_0=0$. So, by conjugating $\Gamma$ if necessary we may assume that $\gamma_1=(0,t_0,0,z')$. Moreover, the assumption that $\Gamma_0\subset Z$ implies that $\widetilde{\psi}(\Gamma)$ is abelian. In particular, $\widetilde{\psi}(\Gamma)$ is included in the centralizer of $\widetilde{\psi}(\gamma_1)=(0,t_0,0)$ in $G_\bX/Z$. A simple computation shows that the latter (for $t_0\neq 0$!) is exactly $\R\times \SO^{\circ}(1,1)$. Thus, $\Gamma\subset \R\times \SO^{\circ}(1,1)\times Z$. 
    
    Now consider the $\Gamma$-invariant map $\M\xrightarrow[]{D} \bX\xrightarrow[]{x^2-y^2} \R$, thus it is well defined on $\pi_1(M)\backslash\M=M$. However, this map has an open image, since $D(\M)$ is an open domain of $\bX$, which is a contradiction with the compactness of $M$. 
 \end{proof}

\begin{lemma}[Theorem 1.13 \cite{raghunathan1972discrete}]\label{Lemma: discrete projection}
    Let $G$ be a Lie group and let $\Gamma$ be a cocompact lattice of $G$. Let $H$ be a normal subgroup of $G$ such that $H\cap \Gamma$ is a cocompact lattice of $H$. Then $p(\Gamma)$ is a cocompact lattice of $G/H$, where $p:G\to G/H$ is the natural projection.
\end{lemma}

The goal of the rest of this section is to show completeness of the lightlike leaves. As explained in Paragraph \ref{para: space of leaves str}, the leaves of $\FF_M$ are defined by a non-singular one form $\eta$. Therefore, the leaves are either all closed or all dense.

\subsubsection{Closed lightlike leaves}\label{Subsec: closed leaves}
By Proposition \ref{Prop: space of F-leaves} we see that $D$ induces a covering map from $\F_\M$ to $D(\F_\M)=\Omega \cap \F$ where $\F$ is the $\FF$-leaf containing $D(\F_\M)$. Hence, the leaf $\F_M$ has, in fact, a Kleinian structure for the affine geometry $(\L_u(1,1), \R^3)$. So when the lightlike leaves are closed, it is sufficient to show that compact Kleinian $(\L_u(1,1), \R^3)$-structures are complete, to deduce the completeness of Kleinian $(G_\bX,\bX)$-structures. We state 

\begin{pr}\label{Prop: Kleinian Markus}
   Kleinian  compact $(\L_u(1,1),\R^3)$-manifolds are complete.
\end{pr}
For the proof of the above proposition we will need the following lemma on subgroups of $\Sol$.
\begin{lemma}\label{Lemma: Sol subgroups}
    Let $\Lambda$ be a non-abelian subgroup of $\Sol\cong \SO^{\circ}(1,1)\ltimes \R^2$. Then either $\Lambda$ is a cocompact lattice, hence, isomorphic to $\Z\ltimes \Z^2$ or $\Lambda_0:=\Lambda\cap \R^2$ has an accumulation point. In the latter case, either $\Lambda_0$ is dense in the translations or  $\Lambda$ is a subgroup of a copy of $\Aff(\R)$.
\end{lemma}
\begin{proof}
    We consider the projection of $\Lambda$ to $\SO^{\circ}(1,1)$, we call this projection $\Lambda'$. The group $\Lambda_0:=\Lambda\cap \R^2$ is non-trivial since $\Lambda$ is not abelian. If $\Lambda_0$ is non-discrete, we are done. If not, its rank must be $2$. Indeed, assume $\Lambda_0\cong \Z$, however, $\Lambda'$ acts on $\Lambda_0$, and its action is necessarily contracting or expanding which is a contradiction, since $\Aut(\Z)$ is finite. So, $\Lambda_0$ is of rank $2$. Thus $\Lambda_0$ is a lattice of $\R^2$, so the projection $\Lambda'$ must be discrete, i.e. $\Lambda\cong \Z\ltimes \Z^2$.
\end{proof}
The next fact will be an important tool and will be used frequently in what follows. Before we state it, recall that a non-trivial homothety of  Heisenberg is an automorphism $\Psi_{\lambda}: \Heis_{2n+1}\to \Heis_{2n+1}, (\xi, z)\mapsto (\lambda\xi, \lambda^2 z)$ where $\lambda\in \R^*\setminus \{1\}$.
\begin{fact}\label{Fact: strong Zassenhaus}
    Let $G:=A\ltimes_{\theta} \Heis_{2n+1}$. Assume that $\theta$ commutes with a non-trivial homothety $\Psi_{\lambda}$ of Heisenberg. Denote by $p: G\to A$. Now, let $\Gamma$ be a finitely generated discrete subgroup of $G$, and let $H$ be the closure of $p(\Gamma)$ in $A$. Define the non-discrete part of $\Gamma$ to be $\Gamma_{nd}:= \Gamma\cap p^{-1}(H^o)$. Then $\Gamma_{nd}$ is a cocompact lattice in a connected, closed nilpotent subgroup $N$ of $G$.
\end{fact}
A proof of this fact can be found in the proof of \cite[Theorem 4.1 (i) + Fact 5.3]{hanounah2025topology}, there $H^o$ is denoted by $\Lambda_0$ and the non-discrete part of $\Gamma$ is denoted by $\Gamma^0$. The idea of the proof can also be found in the paper \cite{carriere1989generalisations}.

\begin{proof}[Proof of Proposition \ref{Prop: Kleinian Markus}]
    Let $N$ be such a manifold. Let $(D,\rho)$ be a developing pair for this structure. We denote by $U$ the image of $D$ and $\Gamma_1:=\rho(\pi_1(N))$. Recall the projection $q: \L_u(1,1)\to \SO^{\circ}(1,1)$. We will discuss accordingly the projection $q(\Gamma_1)$.\\

     \textbf{Unipotent holonomy.} In other words, $q(\Gamma_1)$ is trivial. More precisely, it will be reduced to an $(\L_u(2),\R^3)$-structure (in the sense of Section \ref{Section: geo of models}) which is complete due to Section \ref{Section: completeness transvections}, Fact \ref{Fact: Lu structure} or by \cite[Theorem A]{fried1981affine}. Observe that this argument works even in the non-Kleinian case.\\
    
\textbf{Dense hyperbolic parts.} In this case, $q(\Gamma_1)$ is dense. Observe that $\Gamma_1$ is finitely generated, hence, we apply Fact \ref{Fact: strong Zassenhaus} to get that $\Gamma_1$ has a connected nilpotent syndetic hull $L$. Since $q(L)$ is non-trivial, the intersection $L\cap \Heis_5$ must be central, otherwise we would get non-nilpotent subgroups. Therefore, $\Gamma_0:=\Gamma_1\cap \Heis_5\subset Z$ which is a contradiction by Lemma \ref{Lemma: center transversal}.\\

\textbf{Discrete hyperbolic parts.} In this case $q(\Gamma_1)$ is discrete isomorphic to $\Z$. So we can write $\Gamma_1\cong \Z\ltimes \Gamma_0$. Now we discuss the structure of $\Gamma_0$.
\begin{itemize}
    \item \textbf{Non-abelian $\Gamma_0$.} Here in particular, all $\V$-leaves are closed. Observe that the $\V$-leaves are given by the $\R$-action of the $V$-flow which is just the action center of $\Heis_5$. Thus, the $\R$-action commutes with $\Gamma_1$, so it is well defined on $\Gamma_1\backslash U$. However, since $\Gamma_1\cap \R\neq\{0\}$ the $\R$-action induces a (free) $\mathbb{S}^1$-action on the quotient $\Gamma_1\backslash U$.  Hence, $\Gamma_1\backslash U$ is an $\mS^1$-fiber bundle over a compact flat Lorentz surface. However, the latter is necessarily complete. Since $U$ is $\V$-saturated and its projection modulo $\V$ is complete, we deduce that $U=\R^3$.
    \item \textbf{Abelian $\Gamma_0$.} Discrete abelian subgroups of $\Heis_5$ are of rank $\leq 3$. Moreover, since $\Gamma_0$ cannot be central, $\Gamma_0$ is of rank at least $2$. Furthermore we can assume that $\Gamma_0 \cap Z$ is trivial, otherwise, we have the same conclusion as in non-abelian case. If $\Gamma_0$ is of rank $2$, we conclude that $\Gamma_1$ is of rank $3$. By \cite[3.5. Theorem] {goldman1986affine} we conclude the completeness of $N$ due to the fact that $N$ has parallel volume.\\
If $\Gamma_0$ is of  rank $3$, we will show that this leads to a contradiction. Let $r:\L_u(1,1)\to \L_u(1,1)/I\cong \Sol$ where $I\cong \R^3$ is the kernel of the action of $\Heis_5$ on the space of $\V$-leaves. Looking at the intersection of $\Gamma_0$ with $I$, we observe that $\Gamma_0 \cap I$ is either $\Z$ or $\Z^3$ that is $\Gamma_0\subset I$. Indeed, if the intersection is $\Z^2$, the projection $r(\Gamma_1)$ would be isomorphic to $\Z\times \Z$ but with non-trivial projection to the $\SO^{\circ}(1,1)$-factor and non-trivial intersection with the pure translations which is absurd. If the intersection $\Gamma_0\cap I$ is trivial we get that $r(\Gamma_1)$ injects into $\Sol$, thus by Lemma \ref{Lemma: Sol subgroups} $r(\Gamma_1)\cap \R^2$ has an accumulation point. Now we observe:

\begin{lemma}\label{Obser: aff invariant subsets}
   Let $Q$ be a connected open subset of $\Mink^{1,1}$. Assume that $Q$ is invariant by a non-abelian subgroup $J\subset \Aff(\R) \subset \Isom(\Mink^{1,1})$ such that $J\cap \R^2$ is non-discrete. Then $Q$ is either Minkowski or half Minkowski bounded by a lightlike line.   
\end{lemma}
\begin{proof}
  First we observe that any copy of the affine group in $\Isom(\Mink^{1,1})$ preserves a lightlike foliation $\mathcal{L}$ on Minkowski space. By hypothesis $Q$ is open and $J$-invariant, and since $J\cap \R^2$ non-discrete, we deduce that $Q$ is saturated by the lightlike leaves. Indeed, let $l$ be a lightlike leaf such that $l\cap Q\neq \varnothing$. Let $x_1\in l\cap Q$, taking a small ball around $x_1$ and translate it by the dense subgroup $J\cap \R^2$, we conclude that $l\subset Q$. Now, we consider the projection $\pi:\Mink^{1,1}\to \Mink^{1,1}/\mathcal{L}$. The set $\pi(Q)$ is an open connected subset of $\Mink^{1,1}/\mathcal{L}\cong \R$ which is invariant by a non-trivial subgroup of homotheties (given by the linear part of $J$), hence  $\pi(Q)$ is either $\R$ or $]0,+\infty[$. The claim follows.
   \end{proof}
 The projection  $\widehat{U}$ of $U$ modulo $\V$-leaves is $r(\Gamma_1)$-invariant, and by our assumption $r(\Gamma_1)$ has a non-trivial projection to the $\SO^{\circ}(1,1)$-factor. Let $H$ be the closure of $r(\Gamma_1)\cap \R^2$. By Lemma \ref{Lemma: Sol subgroups} we know that $H$ is either $\R^2$ or $\R$ (tangent to a lightlike direction). If $H=\R^2$ then $\widehat{U}$ must be everything. If $H\cong\R$, then $J:=r(\Gamma)$ and $Q:=\widehat{U}$ satisfy the assumptions of Lemma \ref{Obser: aff invariant subsets}. Thus, $\widehat{U}$ equals $\Mink^{1,1}$ or half Minkowski.
 
 We conclude that $U$ is contractible. Since $D$ is a covering map, $D$ is a diffeomorphism onto $U$. This implies that the holonomy representation is injective, in particular, we have $3=\cd(\pi_1(N))=\cd(\Gamma_1)$. This is a contradiction, since if $\Gamma_0$ is of rank $3$, the cohomological dimension of $\Gamma_1$ is $4$. The fact that $\cd(\Gamma_1)=4$ can be justified by first taking a syndetic hull $S_1$ of $\Gamma_1$ in $\L_u(1,1)$, which exists due \cite[Theorem 2]{saitoII}. We observe that $S_1$ has dimension $4$ and is contractible. Thus by \cite[VIII (8.1)]{brown2012cohomology} $\cd(\Gamma_1)=\dim (S_1)=4$.\\

Finally we have that either $\Gamma_0\cap I \cong \Z$ or $\Z^3$. If the intersection $\Gamma_0\cap I$ is $\Z$ 
the $q(\Gamma_1)$-invariance of $I$ would imply that $\Gamma_0\cap I$ is necessarily central which is excluded by our hypothesis  that $\Gamma_0 \cap Z$ is trivial. So we are left only where the $\Gamma_0\subset I$, here we get that $r(\Gamma_1)$ is isomorphic to $\Z$ and up to conjugacy fixes the point $x_0=(0,0)$.  However, this action cannot be cocompact on an open subset of $\Mink^{1,1}$. Indeed, the quadratic form $f(x,y)=x^2-y^2$ is $r(\Gamma_1)$-invariant, and $f(\widehat{U})$ is open. Passing to the quotient $r(\Gamma_1)\backslash\widehat{U}$ the image must be compact in $\R$, which is a contradiction.
\end{itemize} \end{proof}

\subsubsection{Dense lightlike leaves} Here we treat the case where all the leaves are minimal, i.e. dense. Let $p_1:G_\bX\to \R$ and $p_2:G_\bX\to \R\times \SO^{\circ}(1,1)$ denote the natural projections to the $\R$-factor and to $\R\times \SO^{\circ}(1,1)$, respectively. Let $\psi:\L_u(1,1)\to \L_u(1,1)/Z$ be the projection modulo the center. We denote the image $\psi(S)$ simply by $\widehat{S}$. 

\begin{fact}\label{Fact: dense projection}
    The leaves of the $\FF_M$-foliation are dense if and only if the projection of $p_1(\Gamma)$ is dense.
\end{fact}

Due to the fact that the leaves are no longer closed, in order to understand the dynamics of the holonomy we have to describe explicitly the full isometry group of both geometries.\\

\paragraph{\textbf{Description of isometry groups}}
Recall that the identity component of the isometry group of $\bX_\H$ and $\X_\E$ are denoted by $G_\H$ and  $G_\E$, respectively. Both groups are isomorphic to $(\R\times \SO^{\circ}(1,1))\ltimes\Heis_5$, however, with different $\R$-action. Here we give explicitly the action of the $\R$-factor for both geometries.\\
\paragraph{\textit{Hyperbolic Oscillator geometry.}} An arbitrary element of $\R\times \SO^{\circ}(1,1)$ acts on $\Heis_5$ by the exponential of a derivation $L^\H_{s,t}$ written in a standard (or symplectic) basis of $\heis_5$. Namely, in a basis spanned by $(\z, a_1, a_2, a_3, a_4)$ such that the non-trivial commutators are $[a_1, a_3]=[a_2, a_4]=\z$. The center of Heisenberg $\z$ is in the kernel of $L^\H_{s,t}$ and the transversal subspace spanned by $\mathcal{B}=\{a_1,a_2,a_3,a_4\}$ is invariant. Modulo the center and with respect to the basis $\mathcal{B}$, the derivation $L^\H_{s,t}$ takes the following form,
\begin{equation}L^{\H}_{s,t}:=\label{derivation-hyperbolic}
    \begin{pmatrix}
    
    0&t&-s&0\\
    t&0&0&s\\
    -s&0&0&-t\\
    0&s&-t&0\\
\end{pmatrix}, \ s,t \in \R.\\
\end{equation}
 
\paragraph{\textit{Elliptic geometry.}} Here with same framework as above we have that an arbitrary element of $\R\times \SO^{\circ}(1,1)$ acts on $\Heis_5/Z$ via the exponential of a derivation that takes the following form with respect to the basis $\mathcal{B}$
\begin{equation}
    L^{\E}_{s,t}:=\label{derivation-elliptic}
    \begin{pmatrix}
   
    0&t&-s&0\\
    t&0&0&s\\
    s&0&0&-t\\
    0&-s&-t&0\\
\end{pmatrix}, \ s,t\in \R.
\end{equation}

The $\SO^{\circ}(1,1)$-factor in both cases acts via the exponential of a derivation $L^{\H}_{0,t}$ and $L^{\E}_{0,t}$, respectively. For more details on the isometry groups see \cite[Proposition 4.19]{Matea} or \cite[Proposition 3.9]{kath2024pseudo}.
\begin{remark}[$G_\E$ is not completely solvable]\label{Remark: EV of L_{s,t}}
    In contrast to the hyperbolic geometry, the group of the elliptic geometry is not completely solvable. In fact one can see that the eigenvalues of  the map induced by $L^\E_{s,t}$  on $\heis_5$ modulo its center are $\{\pm t\pm is\}$.
\end{remark}

\begin{remark}\label{Remark: unimodular invariant plane implies hyperbolic geometry}
    A derivation $L=L^\E_{s,t}$ with $s\neq0, t\neq0$, cannot preserve a plane $\widehat{\p}\subset\heis_5/\z$ such that the induced action of $L$ on $\widehat{\p}$ is unimodular, i.e. $\tr(L_|\widehat{\p})=0$.
\end{remark}

\begin{proof}
     Assume such $\widehat{\p}$ exists. By the previous remark in order for the $L$-action to be unimodular on $\widehat{\p}$, $L$ must have two antipodal eigenvalues $\lambda$ and $-\lambda$. However, the derivation $L$ has real coefficients, thus the eigenvalues of $L,$ when restricted to the real plane $\widehat{\p}$, are complex conjugate. So, $-\lambda=\overline{\lambda}$, which implies that $\lambda$ is real. This is a contradiction.
\end{proof}
Since we gave the exact form of the derivation, we can show now the following  lemma. 
\begin{lemma}\label{Lemma: invariant unimodular plane}

 Let $\p\subset\heis_5$ be a proper subalgebra with $\dim \widehat{\p}=2$.
    Assume that $\widehat{\p}$ is invariant under an element $L:=L^\H_{s,t}$, where $s\neq 0$. Assume that the induced action of $L$ is unimodular. Then $\p\cong\heis_3$.
\end{lemma}

\begin{proof}
    Up to scale we can assume $s=1$. The eigenvalues  with the corresponding  eigenvectors  are of the form  \begin{equation}
    \left[\begin{array}{c}
-1-t 
\\
 -1+t 
\\
 1-t 
\\
 1+t 
\end{array}\right],\left[\begin{array}{cccc}
1 & -1 & -1 & 1 
\\
 -1 & 1 & -1 & 1 
\\
 1 & 1 & -1 & -1 
\\
 1 & 1 & 1 & 1 
\end{array}\right]
   \end{equation} 

The Lie algebra $\widehat{\p}$ is spanned by two eigenvectors. Since the action is unimodular we have to consider eigenvectors with opposite eigenvalues. For these vectors we see clearly that their bracket seen as elements in $\heis_5$ is non-zero.
\end{proof}
Coming back to our problem, namely, to show that the lightlike leaves of a compact $(G_\X,\X)$-manifold are complete. As in the beginning of proof of Proposition \ref{Prop: Kleinian Markus}, we can reduce our problem to the case where $q(\Gamma_1)$ is discrete isomorphic to $\Z$. 

The goal of the rest of this subsection is to show that the latter case cannot happen under our standing assumption that $p_1(\Gamma)$ is dense. So assume that $q(\Gamma_1)$ is isomorphic to $\Z$.\\  

Let $H$ be the topological closure of $p_2(\Gamma)$ in $\R\times \SO^{\circ}(1,1)$.
\begin{lemma}\label{Lemma: non-discrete p2 projection}
     The group $H$ contains a one-parameter group transverse to the $\SO^{\circ}(1,1)$-factor.
\end{lemma}
\begin{proof}
   Assume that $H$ were discrete. Since we know that $H\cap \SO^{\circ}(1,1)=q(\Gamma_1)=\Z$, discreteness of $H$ implies that it projects discretely to the $\R$-factor by Lemma \ref{Lemma: discrete projection}. This is absurd due to our hypothesis. Hence, we get that $H$ is either $\R\times \SO^{\circ}(1,1)$ or contains a unique one-parameter group, the latter cannot be in $\SO^{\circ}(1,1)$, since otherwise, the projection $q(\Gamma_1)$ is dense, which is absurd.
\end{proof}
\begin{remark}\label{remark: lattice G_X}
    Let $\Gamma\subset G_\X$ be a discrete subgroup such that $p_1(\Gamma)$ is dense and $q(\Gamma_1)$ is isomorphic to $\Z$. Then $\Gamma$ cannot be a cocompact lattice of $G_\X$.
\end{remark}
\begin{proof}
    Assuming the converse, $\Gamma_0=\Gamma\cap \Heis_5$ is a cocompact lattice of $\Heis_5$,  since $\Heis_5$ is the nilradical of $G_\X$ \cite[Corollary 3.5]{raghunathan1972discrete}. So, $p_2(\Gamma)$ is discrete. This contradicts Lemma \ref{Lemma: non-discrete p2 projection}.
\end{proof}

\begin{remark}\label{remark: dimension S_0}
    Let $S_0$ be the syndetic hull of $\Gamma_0$ in $\Heis_5$. Then $ \widehat{S}_0$ is an $\SO^{\circ}(1,1)$-invariant plane, with a unimodular $\SO^{\circ}(1,1)$-action.
\end{remark}
\begin{proof}
    Let $S_1$ be the syndetic hull of $\Gamma_1$ in $\L_u(1,1)$. Consider the intersection $S_1\cap \Heis_5$. If $S_1\cap \Heis_5\subset Z$ we would get that $\Gamma_0\subset Z$, which is absurd by  Lemma \ref{Lemma: center transversal}. Hence, $S_1\cap \Heis_5$ is the nilradical of $S_1$, therefore $\Gamma_1 \cap S_1\cap \Heis_5=\Gamma_0$ is a cocompact lattice of $S_1\cap \Heis_5$. Since $\Gamma_0$ has a unique Malcev closure in $\Heis_5$, we deduce that $S_1\cap \Heis_5=S_0$. In particular, we get that the projection $\widehat{S}_0$ is a non-trivial subgroup of $\Heis_5/Z$. Since $S_1$ is unimodular we  see that its projection modulo center is unimodular too. However, due to the fact that the induced action of $\SO^{\circ}(1,1)$ on $\Heis_5/Z$ is hyperbolic this implies that $\dim(\widehat{S}_0)=2$ or $\dim(\widehat{S}_0)=4$. We claim that the latter case cannot produce.

    Indeed, if $\dim(\widehat{S}_0)=4$, then we get that $\Gamma_1$ is a cocompact lattice in $\L_u(1,1)$, which would imply that $\Gamma$ is a cocompact lattice of $G_\X$. This contradicts Remark \ref{remark: lattice G_X}.
\end{proof}

  \textbf{Final contradiction.} By Lemma \ref{Lemma: non-discrete p2 projection}, $H$ contains at least a one-parameter group transverse to the $\SO^{\circ}(1,1)$-factor. In addition we always have elements in $\SO^{\circ}(1,1)$ that come from the projection $q(\Gamma_1)$. Since elements of $q(\Gamma_1)$ act by hyperbolic transformations on $\widehat{S}_0$, we see that $\widehat{S}_0$ is even $\SO^{\circ}(1,1)$-invariant. Hence, $\widehat{S}_0$ is, in fact,  $\R\times \SO^{\circ}(1,1)$-invariant. By Lemma \ref{remark: dimension S_0}  we know that $\dim\widehat{S}_0=2$. Hence, by Remark \ref{Remark: unimodular invariant plane implies hyperbolic geometry}, we get $\X=\X_\H$. In particular, $G_\X$ is completely solvable.
     
      Taking a syndetic hull $S$ of $\Gamma$,  again by uniqueness of the Malcev closure, we get $S_0\subset S':=S\cap \Heis_5$. Thus, one gets $\dim(\widehat{S'})\geq \dim(\widehat{S}_0)=2$, and due to unimodularity of $S$ the dimension of $\widehat{S'}$ is either $2$ or $4$. Since $S'$ is normal in $S$, if $\dim(\widehat{S'})=2$, it follows by Lemma \ref{Lemma: invariant unimodular plane}, that $S'\cong \Heis_3$. However, $S_0$ is a (connected) subgroup of $S'$ with $\widehat{S}'=\widehat{S}_0$. Hence, we get  $S'=S_0$. Therefore, $\Gamma_0$ is a cocompact lattice of $S'$. In particular, $p_2(\Gamma)\subset S/S'\cong \R\times \SO^\circ(1,1)$ is discrete, contradicting again Lemma \ref{Lemma: non-discrete p2 projection}. Finally, if $\dim(\widehat{S'})=4$, then $\Gamma$ would be a lattice in $G_\bX$ which contradicts Remark \ref{remark: lattice G_X}.

    \begin{cor}
        Let $M$ be a Kleinian $(G_\X,\X)$-manifold. Assume that the flat lightlike leaves are dense. Then $M$ is modeled on the hyperbolic oscillator geometry. Moreover, the holonomy of the leaves is unipotent, i.e. takes values in the unipotent radical of $\L_u(1,1)$ which is $\Heis_5$.
    \end{cor}

\subsection{Summary}
Let $M$ be a Kleinian compact $(G_\mathbf{X}, \mathbf{X})$-manifold, for $\X\in \{\X_\H, \X_\E\}$. We have shown that if the $\FF_M$-leaves are closed (which is always the case for the elliptic geometry), then they are Kleinian. By Proposition \ref{Prop: Kleinian Markus} it follows that they are complete. Then by \ref{Cor: leaf complete iff M complete} $M$ itself is complete. When the $\FF_M$-leaves are dense we deduced that the leaf holonomy is unipotent, i.e. $q(\Gamma_1)$ must be trivial. Moreover, this can only produce for the hyperbolic oscillator geometry. In particular, $M$ is an $(\L_u(2),\R^3)$-manifold. Therefore $M$ is complete by Theorem \ref{Theorem: complete transvection}. In other words we have:

\begin{cor}\label{Cor: completeness 4D}
    Kleinian compact  $(G_\mathbf{X}, \mathbf{X})$-manifolds are $(G_\mathbf{X}, \mathbf{X})$-complete.
\end{cor}

\section{Geometric classification}\label{Section: topology}
 \subsection{The hyperbolic oscillator geometry}\label{Subsection: splitosc geo}
  Kleinian compact $(G_\mathbf{X}, \mathbf{X})$-manifolds  are complete, i.e. $D$ is a bijective map onto $\bX$, by Corollary \ref{Cor: completeness 4D}.  Hence, the holonomy group $\Gamma$ is a discrete subgroup of $G_\bX$ that acts properly, cocompactly and freely on $\bX$. In this subsection we consider $\bX=\bX_\H$, the hyperbolic oscillator group with its bi-invariant metric and we classify Kleinian compact models in the following sense.

\begin{pr}\label{Prop: Standrad}
Let $M$ be a compact complete $(G_\H, \mathbf{X}_\H)$-manifold. Then the subgroup $\Gamma=\rho(\pi_1(M))$ has a syndetic hull,  i.e. there exists a closed connected subgroup $S$ of $G_\H$ that contains $\Gamma$ as a cocompact lattice. Moreover, $S$ is up to isomorphism either $\R\times \Heis_3$ or $\Osc_s\cong \R\ltimes \Heis_3$, where in both cases $\Heis_3\subset \Heis_5$. In particular, $\Gamma\cong \Z\ltimes \Gamma_0$ where $\Gamma_0$  is a cocompact lattice $\Heis_3$ or $\Gamma$ is isomorphic to a cocompact lattice of $\R\times \Heis_3$.
\end{pr}

\begin{proof}[Proof of Proposition $\ref{Prop: Standrad}$]
    Let $S$ be the syndetic hull of $\Gamma$, due to \cite[Theorem 2]{saitoII}. The Lie group $S$ is contractible, hence by \cite[VIII (8.1)]{brown2012cohomology} we get that $\dim S= 4$. 
    
    First we show that the projection of $S$ to $\R\times \SO^{\circ}(1,1)$ is one-dimensional. Assume that $S$ has a two-dimensional projection, so, surjective onto $\R\times\SO^{\circ}(1,1)$. Hence, we get a two-dimensional connected subgroup $S_0:=S\cap \Heis_5 \cong \R^2$. 
     In particular, $S_0$ is an abelian plane preserved by the  $\R \times \SO^{\circ}(1,1)$-action. Moreover, since $S$ is unimodular, the induced $\R \times \SO^{\circ}(1,1)$-action on ${S_0}$ is also unimodular. In particular, we see that the projection $\widehat{S}_0$ must be $2$-dimensional. However, using Lemma \ref{Lemma: invariant unimodular plane} for $\p=\s_0$, we see that this is absurd. Therefore, the projection of $S$ to $\R\times \SO^{\circ}(1,1)$ is of dimension $1$. Moreover, the projection to the $\R$-factor is non-trivial, since $\Gamma\backslash \bX_\H$ is compact. Thus, $S\cong \R\ltimes S_0$ with $\dim(S_0)=3$ (recall that $S_0$ is connected). Again by Lemma \ref{Lemma: invariant unimodular plane} we see that $S_0$ is necessarily isomorphic to $\Heis_3$ (observe that the condition of $s\neq0$ in the lemma is satisfied). Moreover, the induced $\R$-action on $\widehat{S}_0$ is either hyperbolic or trivial. When the induced $\R$-action  is hyperbolic $S$ is isomorphic to the hyperbolic oscillator. In the other case $S$ is isomorphic to $\R\times\Heis_3$. 
\end{proof}

\begin{cor}\label{Cor: Top clas I}
    Let $M$ be a compact complete $(G_\H, \mathbf{X}_\H)$-manifold. Then  geometrically, $M$ is a fiber bundle over $\mathbb{S}^1$ where the fibers are the lightlike leaves which are compact quotient of Heisenberg (i.e. nilmanifolds) with hyperbolic monodromy, or $M$ is a nilmanifold covered by $\R\times \Heis_3$.  In the second case the lightlike leaves could be dense in $M$.
\end{cor}

Now we turn to compact manifolds modeled on the ``elliptic geometry'' $(G_\E,\mathbf{X}_\E)$.
\subsection{The $(G_\E,\mathbf{X}_\E)$-geometry}\label{Subsection: X_E}
This subsection is devoted to understand compact complete manifolds $M$ which are locally isometric to the space $\bX_\E$.  

As before, we get that $\Gamma=\rho(\pi_1(M))$ is a discrete subgroup of $G_\E$ that acts properly, cocompactly and freely on $\bX_\E$. Moreover, due to contractibility of $\bX_\E$, we deduce that $\mathsf{cd}(\Gamma)=4$. As in Remark \ref{Remark: EV of L_{s,t}} we know that $G_\E$ is not completely solvable we cannot apply \cite{saitoII} to  get a syndetic hull for $\Gamma$. Therefore, in this case we have to analyze the projection of $\Gamma$ to $\R\times \SO^{\circ}(1,1)$ case by case.

 We show:
 Recall the natural projection $p_2: G_\E\to \R\times \SO^{\circ}(1,1)$.
\begin{pr}\label{Prop: X_{E}}
    Let $M$ be a compact complete $(G_\E, \bX_\E)$-manifold. Then the projection $p_2(\Gamma)$ of $\Gamma$ to $\R\times\SO^{\circ}(1,1)$ is discrete isomorphic to $\Z$. Moreover, (up to finite index) $\Gamma\cong \Z\ltimes\Z^3$ or $\Z\ltimes \Gamma_0$, where $\Gamma_0$ is a cocompact lattice of $\Heis_3$.
     In particular, the lightlike leaves are always closed.
\end{pr}

\begin{proof}  We distinguish two natural cases:\\
\textbf{Non discrete projection.}
\begin{itemize}
    \item \textbf{Connected closure.} Hence, $\overline{p_2(\Gamma)}=\R^2$ or $\R$. Using Fact \ref{Fact: strong Zassenhaus} we conclude that $\Gamma$ has a nilpotent syndetic hull $N$. Hence, any derivation $L^\E_{s,t}$ generating a one-parameter group in the projection has a non-trivial kernel (modulo center), this is a contradiction by Remark \ref{Remark: EV of L_{s,t}}.
    
   \item \textbf{Non-connected closure}. $\overline{p_2(\Gamma)}\cong \R\times \Z$. Now let $\Gamma'$ be the non-discrete part of $\Gamma$, in the sense of Fact \ref{Fact: strong Zassenhaus}. Namely, $\Gamma'$ is the intersection of $\Gamma$ with the inverse image of the $\R$-factor  of $\overline{p_2(\Gamma)}$. Hence, one gets  the following exact sequence \begin{equation*}
       0\to\Gamma'\to \Gamma\to \Z\to0.
       \end{equation*}
    Since $\Gamma$ is virtually polycyclic (see \cite[Lemma 2.2]{milnor1977fundamental}) we have the equality on the cohomological dimension $\cd(\Gamma)=\cd(\Z)+\cd(\Gamma')$ (\cite[8.8 Lemma 8]{gruenberg2006cohomological}). So, $\cd(\Gamma')=3$. Hence, by Fact \ref{Fact: strong Zassenhaus} $\Gamma'$ has a nilpotent syndetic hull $N'$ in $G_\E$. Similarly, $\dim(N')=3$. However, this would imply that  the derivation $L^\E_{s,t}$  that generates the  $\R$-factor of the projection of $\Gamma'$ has a non-trivial kernel, which is absurd again by Remark \ref{Remark: EV of L_{s,t}}.
\end{itemize}
 \textbf{Discrete projection.}
 \begin{itemize}
     \item \textbf{Rank $2$ projection}.   $\overline{p_2(\Gamma)}\cong \Z\times \Z \subset \R\times \SO^{\circ}(1,1)$. Since $\Gamma$ is polycyclic, we get that $\cd(\Gamma_0)=2$, where $\Gamma_0=\Gamma\cap \Heis_5$. Let $\hat{\gamma}_1$ and $\hat{\gamma}_2$ be two elements of $\Gamma$ that project to two generators of $\Z\times \Z$. Let $S_0$ be syndetic hull (it is just the linear span of $\Gamma_0$) of $\Gamma_0$ in $\Heis_5$. The action of  $\hat{\gamma}_i$ by conjugation preserves $\Gamma_0$. Therefore it preserves $S_0$. Hence, their action on $S_0$  must be unimodular. Here we distinguish two subcases:\\
     \textbf{Case $Z\subset S_0$.} The induced action of $\hat{\gamma}_i$ on $\widehat{S}_0$ is given by $\exp(L^\E_{s_i,t_i})$. Moreover, since $Z$ is fixed and the action is unimodular, the action  on the one-dimensional quotient space $\widehat{S}_0$ must be trivial  for $\exp(L^\E_{s_i,t_i})$ for $i\in \{1,2\}$. By Remark \ref{Remark: EV of L_{s,t}} we get that $t_i=0$, which is a contradiction.\\
     \textbf{Case $Z\pitchfork S_0$.} Since the induced action of $\hat{\gamma}_i$ on $\Heis_5/Z$ is unimodular and semisimple, then the induced action on $\widehat{S}_0$ is just a semisimple element of $\operatorname{SL}_2(\R)$. Hence, the induced action of each $\hat{\gamma}_i$ is by a rotation (maybe trivial rotation), or by a hyperbolic transformation. However, since the $\hat{\gamma}_i$ commute, they must both be of the same type. In particular, when they both act by rotation we get $t_1=t_2=0$, and when the action is hyperbolic we get that $s_1,s_2\in \pi\Z$. The first case is clearly a contradiction. However, in the second case we get a non-trivial element $A\in \SO^{\circ}(1,1)\cap p_2(\Gamma)$. Namely, there are $m,n\in \Z$ such that $\log(A)=n{L}^\E_{s_2,t_2}-m{L}^\E_{s_1,t_1}\in \so(1,1)$. In other words, there is an element $\hat{\gamma}_3\in \Gamma$ such that  $p_2(\hat{\gamma}_3)\in  \SO^{\circ}(1,1)$. This implies that $\cd(\Gamma_1)=3$. Since  $\L_u(1,1)$ is completely solvable, $\Gamma_1$ has a syndetic hull $S_1$. One can see clearly that $S_1\cong \Sol$. Moreover, since $\Gamma_1$ acts properly, the $S_1$-action is proper. Now we show,

     \begin{sublem}
         There is no group $S_1\cong \Sol\subset \L_u(1,1)\subset G_{\E}$ that acts properly $\X_\E$.
     \end{sublem}
     \begin{proof}
         Consider $S_0=S_1\cap \Heis_5$. The abelian subgroup $S_0$ is spanned by $v_1+\alpha z$ and $v_2+\beta z$, where $\alpha, \beta \in \R$. Up to conjugacy by an element $h\in \Heis_5$, we can assume that $\alpha=\beta=0$. Now we show that $S_0$ intersects non-trivially the stabilizer of $p_s$ for some $s$. Indeed, in the basis $(a_1,a_2,a_3,a_4)$, we have that $\Stab(p_s)=\Span\la w_1,w_2\ra$ where $w_1=(\cos s,0,\sin s,0)$ and $w_2=(0, \cos s, 0, -\sin s)$. On the other hand, since $v_1,v_2$  are $\SO^\circ(1,1)$-invariant, hence $v_1=(x,0,x,0)$ and $v_2=(0,y,0,y)$. The non-trivial intersection of $S_0$ with the stabilizer is reduced to the system of equations \begin{equation*}
             \alpha \cos s=\alpha \sin s,\ \beta \cos s=-\beta \sin s, \ \text{with} \  (\alpha,\beta)\neq (0,0)
         \end{equation*} The latter system can always be solved for $s$. Since $S_0$ contains a non-compact subgroup in the stabilizer of a point, the action of $S_0$ cannot be proper. 
     \end{proof}
     The above lemma leads to a contradiction with the fact that $\Gamma_1$ acts properly on $\X_\E$.
     \item \textbf{Rank $1$ projection}. $\overline{p_2(\Gamma)}\cong \Z \subset \R\times \SO^{\circ}(1,1)$. This implies that $\mathsf{cd}(\Gamma_0)=3$. Hence, the syndetic hull $S_0$ of $\Gamma_0$ has two possibilities: $\R^3$ or  $\Heis_3\hookrightarrow \Heis_5$.  In both cases, $S_0$ contains $Z$. 
      Take an element $\hat{\gamma}\in \Gamma$ that projects to a generator of $p_2(\Gamma)$. Then  the action  of $\hat{\gamma}$ preserves $\Gamma_0$, thus $S_0$, so its action on $S_0$ is unimodular. Moreover, since the action is semisimple on Heisenberg modulo the center, the induced action on $\widehat{S}_0$ is either by a rotation or by a hyperbolic matrix. We conclude, that virtually $\Gamma\cong \Z\ltimes\Z^3$ or $\Gamma\cong \Z\ltimes\Gamma_0$ where $\Gamma_0$ is a lattice of $\Heis_3$.
 \end{itemize} \end{proof}

 Next we give a list of examples of the cases that are not ruled out in the above proposition.
\begin{examples}
     [\textbf{The Lightlike leaves are covered by $S_0$ isomorphic to $\R^3$ or $\Heis_3$.}] We construct a discrete group $\Gamma\subset G_\E$ with rank $1$ projection, such that $\Gamma$ acts properly and cocompactly on $\X_\E$. 

We  will first construct a group isomorphic to $S_0\subset \L_u(1,1)$ acting simply transitively on each lightlike leaf. We claim that this is equivalent to show that \begin{equation}\label{Eq: proper action}
    S_0 \cap \Stab_{\Heis_5}(p_s)=\{0\}, \ \text{for all} \ s\in \R.
\end{equation} Where $p_s=(0,0,0,s)$ written in the $\X_\E$-coordinates see \cite[Proposition 3.7]{kath2024pseudo}. Indeed, assume that $S_0 \cap \Stab(p_s)=\{0\}$, for all $s\in \R$. Since $S_0$ preserves each leaf, the assumption is exactly that the orbit $\mathcal{O}_s=S_0 \cdot p_s$ is open in the leaf $\F(p_s)$. However, the action of $S_0$ on each leaf preserves the flat affine structure of the leaves. Therefore, we can pullback the affine structure from the open orbit $\mathcal{O}_s$ to $S_0$, giving rise to a left invariant affine structure on $S_0$. However, since the affine structure is unimodular, i.e. $\L_u(1,1)\subset \operatorname{SL}_3(\R)$, such a structure is complete by \cite[Corollary pp. 186]{goldman1986affine}. Now taking $\Gamma\cong \la \hat{\gamma}\ra \times \Gamma_0$ where $\hat{\gamma}=(2\pi,\operatorname{Id})$ and  $\Gamma_0\subset S_0$ is a cocompact lattice. This gives rise to a discrete subgroup that acts properly, cocompactly  and freely on $\X_\E$. 

Now we construct $S_0\cong\R^3$. As $S_0$ is connected, it is enough to determine its Lie algebra. We claim that the abelian Lie subalgebra $\s_0:=\Span \la a_1+a_4,a_2+a_4,\z\ra $ satisfies Equation (\ref{Eq: proper action}). Indeed, $\Stab_{\Heis_5}(p_s)=\Span\la w_1,w_2\ra$ where $w_1=(\cos s,0,\sin s,0)$ and $w_2=(0, \cos s, 0, -\sin s)$. If $\Span\la w_1,w_2\ra \cap S_0 \neq \{0\}$, then there exists $(\alpha, \beta)\neq (0,0)$ such that $\alpha v_1+\beta v_2\in S_0$. This would imply that $e^{is}\begin{pmatrix}
    \alpha\\
    \beta
\end{pmatrix}=0$, which is absurd. It is worth noting  here that $M=\Gamma \backslash \X_\E$ is, in fact, isometric to a $\T^4$.

For an example where $S_0$ isomorphic to $\Heis_3$, we consider the Lie subalgebra $\s_0$, which is spanned by $\{a_1+a_3+a_4, a_2+2a_3+a_4,\z\}$. We argue similarly to the previous case and show that Equation (\ref{Eq: proper action}) is satisfied.

\end{examples}

\begin{cor}\label{Cor: Top clas II}
   Let $M$ be a  compact complete $(G_{{\E}}, \bX_\E)$-manifold. Then $M$ is geometrically (up to finite cover), a fiber bundle over $\mS^1$  where the fibers (which are the lightlike leaves) are either isometric to \begin{itemize}
       \item The flat three torus.
       \item A nilmanifold $N$, where $N$ is a compact quotient of $\Heis_3$. 
       
   \end{itemize}
       
   Conversely, there are examples realizing the above cases.
\end{cor}
\begin{proof}
    The geometric classification is a direct consequence of Proposition  \ref{Prop: X_{E}}. The examples are given above.
\end{proof}

\section{Further discussion: Dynamics of the parallel flow}\label{Section: dynamics}

Now we make some remarks concerning the dynamics of the parallel flow of compact locally symmetric rank $1$ pp-waves, highlighting a fundamental difference between the higher signature case and the Lorentz case. We recall that the action of a Killing field on a pseudo-Riemannian manifold $M$ is called equicontinuous if its closure (with respect to the Lie group topology) in the isometry group of $M$ is compact. 
We state first the following fact for locally homogeneous plane waves, 
\begin{fact}[Proposition 8.2 \cite{kath-CW}]
    Let $M$ be compact Cahen--Wallach. Then the parallel flow is periodic, i.e. the parallel flow is a $\mS^1$-action.
\end{fact}
The hyperbolic oscillator geometry has a similar behavior to the Lorentz case, namely:
\begin{pr}

    Let $M$ be a compact complete $(G_\X,\mathbf{X})$-manifold. Then the action of the parallel flow is equicontinuous. Moreover, for the hyperbolic oscillator geometry it is always periodic.
\end{pr}
\begin{proof}
\textbf{Hyperbolic oscillator geometry.} The isometry group of $M$ is given by $\Gamma\backslash N_{G_\H}(\Gamma)$, where $\Gamma$ is the $\pi_1(M)$. Let $S$ be a syndetic hull of $\Gamma$. By Proposition \ref{Prop: Standrad}, we have that $S$ is either $\R\times \Heis_3$ or $\Osc_s=\R\ltimes \Heis_3$. We claim that $[\Gamma,\Gamma]\cap Z$ is non-trivial. Indeed, in both cases $[\Gamma, \Gamma]\subset \Heis_3$. If $S=\Osc_s$, then $[\Gamma, \Gamma]=\Gamma\cap \Heis_3$ is a lattice of $\Heis_3$, therefore it must intersect the center $Z$. In the case $S=\R\times \Heis_3$ we see that $[\Gamma,\Gamma]\subset Z \subset \Heis_3$. Moreover $[\Gamma,\Gamma]\cap Z$ cannot be trivial, because otherwise $\Gamma$ would be an abelian lattice of the nilpotent group $S$, hence $S$ itself would be abelian, which is absurd.\\
\textbf{Elliptic geometry.} The holonomy of the flat leaves is unipotent and is contained in $\Heis_5$. Let  $S_0$ be a syndetic hull of $\Gamma_0$. Since $\cd(\Gamma_0)=3=\dim S_0$, we see that $S_0$  contains the center of $\Heis_5$. The group $\Gamma$ is contained in $\Z\ltimes S_0$. Therefore, the topological closure $\overline{\Gamma Z}$ is contained in $\Gamma S_0$, because $\Gamma\backslash\Gamma S_0\cong \Gamma_0\backslash S_0$ is compact the result follows.
\end{proof}

\appendix
\section{Lie algebra of the transvection group}\label{Appendix}

\begin{pr}
      Let $\bX$ be a symmetric rank $1$ pp-wave of signature $(2,n)$ with $\dim(\bX)=n+2$.  Then the Lie algebra of the transvection group of $\bX$ is given by $\widehat{\g}_\bX\cong \R\ltimes_L\heis_{2n+1}$, where $\R$ acts trivially on the center of $\heis_{2n+1}$ via a non-inner derivation $L$. Moreover, the symmetric space  $\bX$ is isometric to the homogeneous quotient $\G_\bX/A^+$ where $A^+$ is a rank $n$ abelian subgroup of $\Heis_{2n+1}$ transversal to the center. The center of Heisenberg generates the lightlike vector field $V$ and the lightlike leaves are the $\Heis_{2n+1}$-orbits.
\end{pr}
\begin{proof}
    Recall that since $\bX$ is a symmetric space we can extend (uniquely) the metric on $\bX$ to an $\ad$-invariant metric $\langle \cdot , \cdot \rangle$ and an involution $\Theta$  on the Lie algebra of its transvection group. Moreover, $\langle \cdot , \cdot \rangle$ is $\Theta$-invariant. Hence, we have the decomposition $\widehat{\g}_\bX=\widehat{\g}_\bX^+\oplus\widehat{\g}_\bX^-$ with respect to the $\pm 1$ eigenvalues of $\Theta$.  The isotropy algebra $\h$ is given by $\widehat{\g}_\bX^+$ and $\widehat{\g}_\X^-$ is identified with the tangent space of $\bX$ at the base point. 

Now,  the vector field $V$ at the base point is an element of $\widehat{\g}_\X^-$. Hence, we can decompose $\widehat{\g}_\bX=\R L\oplus V^\perp$, where $L$ is a lightlike vector in $\widehat{\g}_\X^-$. Moreover, we decompose $V^\perp$ as $V^\perp=V\oplus \a$,  where we choose $\a=\Span\{L,V\}^\perp$ (the plane $\Span\{L,V\}$ is Lorentz, $\la L,V\ra=1$). In particular, from the $\Theta$-invariance of $\langle \cdot, \cdot \rangle$ we see that $\a$ is $\Theta$-invariant. We denote $\a^+$ and $\a^-$ the eigenspaces of $\Theta$ restricted to $\a$ with eigenvalues $\pm 1$.  

Observe that $\a$ is $L$-invariant as from the $\ad$-invariance we get $\la[L,\a],L\ra=0$. We claim that $[\a,\a]\subset \R V$. More precisely, for $A_1,A_2\in \a$ we have $[A_1,A_2]=\la L(A_1), A_2\ra V$ where $L(A_1):=[L,A_1]$. Moreover, the $2$-form $\omega:=\la L\cdot, \cdot\ra$ on $\a$ is non-degenerate.\\

\textbf{The derived subalgebra $[V^\perp,V^\perp]$.}  Recall that we have the following fundamental property for the transvection group $[\widehat{\g}_\bX^-,\widehat{\g}_\bX^-]=\widehat{\g}_\bX^+$. It follows in particular, that $[\R L\oplus\a^-\oplus \R V, \R L\oplus\a^-\oplus \R V]=\a^+$. Since $V$ is parallel, it is central. So we get, \begin{equation}\label{Eq: transvection1}
    [\R L\oplus\a^-,\a^-]=\a^+.
\end{equation}
Now we show that $[\a^-,\a^-]=0$. Indeed, $\la[\a^-, \a^-],\a^+\ra=\la[\a^-, \a^-],[\R L\oplus\a^-,\a^-]\ra$. Using the $\ad$-invariance of the metric and the fact that leaves are flat (i.e. $[[\a^-,\a^-],\a^-]=0$), we get that $\la[\a^-, \a^-],\a^+\ra=0$. However, $[\a^-,\a^-]\subset \a^+$, and the metric restricted on $\a^+$ is not degenerate ($\a^+$ is transversal to $V$), we get that $[\a^-,\a^-]=0$. Plugging this in Equation (\ref{Eq: transvection1}) we get that \begin{equation}\label{Eq: transvection2}
    [L,\a^-]=\a^+.
\end{equation}
Using that $[\a^-,\a^-]=0$, we get $\la[\a^-,\a^+], \a^-\ra=0$. Since, $\la[\a^-,\a^+], V \ra=0$, it follows that $[\a^-,\a^+]\subset \a^-\oplus \R V$.  Thus, $[\a^-,\a^+]\subset \R V$ (since the metric on $\a^-$ is non degenerate).\\
Finally we show that $[\a^+, \a^+]=0$. For this we consider, $\la[\a^+,\a^+], \a^+\ra$, however, by Equation (\ref{Eq: transvection2}) we can write  
\begin{equation*}\label{EQ: omega}
    \la[\a^+,\a^+], \a^+\ra=\la[\a^+,\a^+], [L, \a^-]\ra= \la [\a^-,[\a^+,\a^+]], L\ra=0.
\end{equation*}
The last equality above follows from the fact $[\a^+, \a^+]\subset \a^-$ together with $[\a^-,\a^+]=0$. Thus it follows that $[\a^+, \a^+]=0$. So, $[\a,\a]\subset \R V$, which implies that $[V^\perp,V^\perp]\subset \R V$.\\

\textbf{Heisenberg structure.} Here we show that the $2$-form $\omega=\la L\cdot, \cdot\ra$ is non-degenerate and for $A_1,A_2\in \a$ we have $[A_1,A_2]=\omega(A_1,A_2)V$. The second part of the statement  follows from $[A_1,A_2]\in \R V$ and $\omega(A_1,A_2)=\la L, [A_1,A_2]\ra=\la [L,A_1],A_2\ra$. Now, assume that $\omega$ were degenerate, so there exists a non-trivial element $W\in\a$ such that $\omega(W,\cdot)\equiv 0$.

First, we observe that  $L(W)^\perp=\widehat{\g}_\bX$, hence $L(W)=0$. We write $W=W^++W^-$ where $W^\pm\in \a^\pm$. Expressing $L(W)=L(W^+)+L(W^-)=0$, since $L(W^\pm)\in \a^\mp$ we get that $L(W^\pm)=0$. We claim that $W^\pm=0$. In fact, if $W^-\neq0$, then it defines a parallel vector field because it is  invariant by the isotropy. Namely, we have \begin{equation*}
    [\a^+, W^-]=\la L(W^-),\a^+\ra V=0.
\end{equation*} The parallel vector field $W^-$ cannot be lightlike, since by assumption $V$ is unique (up to scale). However, if $W^-$ is not lightlike it means that $\bX$ is reducible, which is a contradiction. Using Equation (\ref{Eq: transvection2}), so we get $\la W^+, A_2\ra=\la W^+, [L, A_1]\ra=-\la[L, W^+], A_1\ra=0$, for $A_2\in \a^+, A_1\in \a^-$. Thus, $W^+ \perp \a^+$, this means that $W^+=0$.
\end{proof}

\bibliographystyle{plain}
\bibliography{Bibliography}

\end{document}